\documentclass[11pt]{ip-journal}

\usepackage[colorlinks=false,pagebackref,hyperindex]{hyperref}
\usepackage[english]{babel} 
\usepackage{latexsym}

\usepackage{graphicx}
\usepackage{graphics} 
\usepackage{fancyhdr}
\usepackage{enumerate}
\usepackage{color} 

\usepackage{amsfonts}  
\usepackage{amsmath}
\usepackage{amssymb} 
\usepackage{amsrefs}
\usepackage{amsthm}
\usepackage{mathrsfs} 

\input xy \xyoption{all}

\DeclareRobustCommand{\SkipTocEntry}[5]{}

\newtheorem{theorem}{Theorem}[section]
\newtheorem{lemma}[theorem]{Lemma}
\newtheorem{conjecture}[theorem]{Conjecture}
\newtheorem{proposition}[theorem]{Proposition}
\newtheorem{corollary}[theorem]{Corollary}
\newtheorem{definition}[theorem]{Definition}
\newtheorem{question}[theorem]{Question}

\theoremstyle{definition}

\newtheorem{remark}[theorem]{Remark}

\DeclareMathOperator{\Vol}{\mathrm{Vol}}
\DeclareMathOperator{\vol}{\mathrm{vol}}
\DeclareMathOperator{\codim}{codim}
\DeclareMathOperator{\diff}{\mathrm{Diff}}

\newcommand{\bQ}{\mathbb{Q}}

\newcommand{\bR}{\mathbb{R}}

\newcommand{\supp}{{\rm Supp}}
\newcommand{\mys}{\mathcal{Y}_s}

\newcommand{\mcy}{\mathcal{Y}}
\newcommand{\mcx}{\mathcal{X}}

\address{Yau Mathematical Sciences Center, 
JingZhai, 
Tsinghua University, 
Hai Dian District, 
Beijing, 
China
100084.} 
\email{birkar@tsinghua.edu.cn}

\address{Department of Mathematics.
Princeton University,
Princeton,
NJ 08540, 
USA.} 
\email{gdi@math.princeton.edu}

\address{Universit\`a degli Studi di Milano,
Dipartimento di Matematica ``F. Enriques'', 
Via Saldini 50, 
20124 Milano (MI), 
Italy.}
\email{roberto.svaldi@unimi.it}

\author{Caucher Birkar, Gabriele Di Cerbo, and Roberto Svaldi}

\title[Boundedness of elliptic Calabi--Yau varieties with section]{Boundedness of elliptic Calabi--Yau varieties with a rational section}

\addtocontents{toc}{\protect\setcounter{tocdepth}{0}}

\begin{document}
\begin{abstract}
We show that for each fixed dimension 
$d\geq 2$, 
the set of 
$d$-dimensional 
klt elliptic varieties with numerically trivial canonical bundle is bounded up to isomorphism in codimension one, provided that the torsion index of the canonical class is bounded and the elliptic fibration admits a rational section.
This case builds on an analogous boundedness result for the set of rationally connected log Calabi-Yau pairs with bounded torsion index. 
In dimension 
$3$, 
we prove the more general statement that the set of 
$\epsilon$-lc pairs 
$(X,B)$ 
with 
$-(K_X +B)$ 
nef and rationally connected 
$X$ 
is bounded up to isomorphism in codimension one.
\end{abstract}



\thanks{
CB was supported by a grant of the Leverhulme Trust and a grant of the Royal Society.
GD was supported by the NSF under grants numbers DMS-1702358 and DMS-1440140 while the author was in residence at the Mathematical Sciences Research Institute in Berkeley, California, during the Spring 2019 semester.
Part of this work was completed during a visit of RS to Princeton University. 
RS would like to thank Princeton University for the hospitality and the nice working environment, and J\'anos Koll\'ar for funding his visit.
RS kindly acknowledges support from Churchill College, Cambridge, the European Union's Horizon 2020 research and innovation programme under the Marie Sk\l{}odowska-Curie grant agreement No. 842071, and from the ``Programma per giovani ricercatori Rita Levi Montalcini''.
}

\maketitle
\tableofcontents

\addtocontents{toc}{\protect\setcounter{tocdepth}{1}}

\section{Introduction}\label{intro}
Throughout this paper, we work over an algebraically closed field
of characteristic 0, for instance, the field of complex numbers $\mathbb{C}$.

The central task in algebraic geometry is the classification of projective varieties.
There are two possible approaches to this end: either by identifying two distinct varieties that are isomorphic or by saying that two distinct varieties are birational equivalent (or simply, birational) if they both possess isomorphic dense open sets.
Birational equivalence, preserving many numerical and geometrical quantities, is a sufficiently coarse equivalence relation in the category of algebraic varieties; at the same time, it is more flexible than the classification by isomorphism type: 
we can modify a given variety as long as a dense open set is left untouched, thus constructing a new variety birational to the original one.
The Minimal Model Program predicts that, up to a special class of birational transformations, each projective variety decomposes into iterated fibrations with general fibres of $3$ basic types:
\begin{itemize}
	\item {\it Fano varieties}: mildly singular varieties with ample anti-canonical bundle;
	\item {\it K-trivial varieties}: mildly singular varieties with torsion canonical bundle;
	\item {\it log canonical models}: mildly singular varieties with ample canonical bundle.
\end{itemize}  
As these classes of varieties constitute the fundamental building blocks in the birational classification of algebraic varieties, the task of understanding their possible algebraic and topological structures is a central one. 

The second part of the classification scheme aims to construct compact {\it moduli spaces} that parametrize all isomorphism classes for the above 3 classes of varieties. 
For this purpose, one big question is whether or not in any given dimension there are just finitely many families of varieties for each of the 3 classes  -- perhaps after fixing certain geometric invariants. 
This property goes under the name of {\it boundedness}, see \S~\ref{bound.sect}.
This is a central question as these 3 classes of varieties are building blocks for many constructions in geometry and theoretical physics.

Showing boundedness for a certain class of varieties is a rather difficult task: in fact, one needs to embed all varieties in the chosen class in a fixed projective space, while, at the same time, controlling the degree of these embeddings. 
There have been some recent extraordinary developments on the study of boundedness for two of the building blocks introduced above:
Hacon--M\textsuperscript{c}Kernan--Xu proved that log canonical models with fixed volume are bounded in fixed dimension, see~\cite{MR3779687}; 
this result also implies the existence of moduli spaces of log canonical models, thanks to work of Koll\'ar--Shepherd-Barron, Alexeev, and others, see~\cite{MR3184176}.
Koll\'ar-Miyaoka-Mori showed in~\cite{MR1189503} that smooth Fano varieties form a bounded family; recently, Birkar proved the boundedness of $d$-dimensional $\epsilon$-klt Fano varieties for fixed $\epsilon >0$ and $d \in \mathbb{N}$, thus, proving the BAB Conjecture, see~\cite{Bir16b}; 
moreover, he also generalised this result to the case of log Fano fibrations,~\cites{1811.10709, 1811.107092}, where some ideas and techniques from~\cites{MR3862064, DCS} are used.

K-trivial varieties, in turn, are not bounded in the category of algebraic varieties: 
for example, projective K3 surfaces form infinitely many algebraic families, due to the surjectivity of the period mapping, see \cite{Huyb_K3}*{Chapter 7}; 
it also well-known that abelian varieties are not bounded, in any fixed dimension. 
Nonetheless, both K3 surfaces and abelian varieties of fixed dimension all fit into a unique topological family, once we consider also the non-algebraic ones. 
In dimension higher than 2, the situation is even more varied. 
In these contexts, the (possible) lack of boundedness can be remedied by fixing a polarization with bounded volume -- a classic approach in the study of moduli of algebraic varieties that do not possess an obvious polarization, cf.~\cite{MR1368632}.
More recently,~\cite{2006.11238}, Birkar has shown that polarised Calabi--Yau varieties and log Calabi--Yau pairs are well-behaved from the point of view of the Minimal Model Program and boundedness questions, much like in the case of Fano varieties or canonical models, despite the lack of uniqueness for the choice of polarization.

\subsection*{Boundedness for elliptic Calabi--Yau varieties} 
A now classical result due to Beauville and Bogomolov,~\cite{MR730926}, shows that, up to an \'etale cover, every smooth proper variety with numerically trivial canonical class can be decomposed into a product of abelian, hyperk\"ahler, and Calabi--Yau manifolds. 
A smooth proper variety $Y$ is Calabi--Yau if it is simply connected, $K_Y \sim 0$ and $H^i(Y, \mathcal{O}_Y)=0$ for $0<i<\dim Y$.

While we know that in any fixed dimension $d\geq 2$, there exists infinitely many algebraic families of abelian or algebraic hyperk\"ahler manifolds, Calabi--Yau varieties constitute a large class of K-trivial varieties for which boundedness is still a hard unresolved question.
\begin{question}
\label{cy.probl}
Fix $d \in \mathbb{N}$. 
Are $d$-dimensional Calabi--Yau manifolds bounded?
\end{question}
\noindent
The only known affirmative answer to the above question is due to Gross,~\cite{MR1272978}, who proved that boundedness holds for projective Calabi--Yau threefolds carrying an elliptic fibration, up to birational equivalence.
Recently, Gross's results was strengthened, by the third-named author together with Filipazzi and Hacon~\cite{2112.01352}, to show that projective elliptic Calabi--Yau threefolds are actually bounded.
These results are in opposition to the two-dimensional case: in fact, there exist infinitely many algebraic families of elliptic K3 surfaces.
Starting from dimension three, theoretical physicists have formulated the expectation, based on the known examples, that Calabi--Yau manifolds with sufficiently large Picard number may be modified birationally to obtain a model that is endowed with a fibration to a lower dimensional variety, see~\cite{Anderson2017}*{Conjecture, page~27}.
In view of this, then, proving boundedness results for elliptically fibred Calabi--Yau manifolds would provide an important step towards answering Question~\ref{cy.probl}, as it would show that the Picard number can be bounded from above.

It is not hard to show that if we require the presence of a section, already in the case of elliptic K3 surfaces, it is possible to construct polarizations with bounded volume, thus proving boundedness for this class of K3 surfaces.
The main goal of this paper is to prove the following generalization to arbitrary dimension of this phenomenon.

\begin{theorem}
\label{ell.cy.thm}
Fix a positive integer $d$.
Then the set of projective varieties $Y$ such that
\begin{enumerate}
    \item $Y$ is a Calabi--Yau manifold of dimension $d$ and
    \item $Y \to X$ is an elliptic fibration with a rational section $X \dashrightarrow Y$
\end{enumerate}
is bounded up to 
flops.
\end{theorem}
\noindent

Theorem~\ref{ell.cy.thm} is an important foundational result: it implies that there are just finitely many families of elliptic Calabi-Yau varieties carrying a rational section, up to birational operations that modify the variety in a subset of codimension at least $2$. 
Such birational operations do not modify excessively the geometry of the variety; for example, the Hodge diamond is left unaltered, as shown in~\cite{MR1672108}. 

In turn, Theorem~\ref{ell.cy.thm} implies that, in any fixed dimension, there are finitely many possible different Hodge diamonds of elliptic Calabi-Yau manifolds admitting a rational section.

\begin{corollary}
\label{hodge}
Fix a positive integer $d$. 
Then there exists a positive integer 
$M_d$ 
such that 
$h^{p,q}(Y)\leq M_d$ 
for any 
$p,q$, 
where 
$Y$ 
is a smooth manifold satisfying the conditions of Theorem~\ref{ell.cy.thm}.
\end{corollary}

The above result is relevant to applications to physics, where F-theory constructions are based on a choice of a Calabi-Yau manifold $Y$ with an elliptic fibration over a base space $X$. 
In this context, physical properties of the model can be translated into the geometry of the elliptically fibred Calabi-Yau and the uniform boundedness of some of their Hodge numbers is a property that has long been sought after in the field, see, for example,~\cite{MR3464638}.

When studying elliptic Calabi--Yau varieties in general, those admitting a rational section play a central role. 
Weierstrass models and Jacobian fibrations always carry a rational section and their boundedness is a fundamental step towards the proof of boundedness of elliptic Calabi--Yau threefolds in Gross' work. 
Thus, Theorem~\ref{ell.cy.thm} represents the realization of this first fundamental step in the generalization of~\cite{MR1272978} to higher dimension.
Theorem~\ref{ell.cy.thm} combined with a careful analysis of the Tate-Shafarevich group for elliptic fibrations should eventually produce a proof of the boundedness of general elliptic Calabi-Yau manifolds.

A similar result to Theorem~\ref{ell.cy.thm} holds for singular K-trivial varieties as long as we bound the torsion index of the canonical divisor.

\begin{theorem}
\label{ell.cy.thm.gen}
Fix positive integers $d, l$.
Then the set of varieties $Y$ such that
\begin{enumerate}
    \item $Y$ is klt and projective of dimension $d$,
    \item $lK_Y \sim 0$,
    \item $Y \to X$ is an elliptic fibration with a rational section $X \dashrightarrow Y$, and 
    \item $X$ is rationally connected
\end{enumerate}
is bounded up to flops.
\end{theorem}

Since in Theorem~\ref{ell.cy.thm.gen} we dropped the assumption on the simple connectedness of $Y$, there is a price to pay in terms of the assumptions we make: namely, we have to assume that the base $X$ of the fibration is rationally connected.
Considering elliptic K-trivial varieties with a rationally connected base is a rather natural restriction: indeed, it is simple to see that when $Y$ is Calabi--Yau then the base of an elliptic fibration is always rationally connected, cf. Corollary~\ref{rc.base.cor}.
More generally, imposing only that $K_Y$ be numerically trivial, it may happen that the base $X$ of the elliptic fibration $Y \to X$ is not rationally connected, as it is already easy to see by taking the product of an elliptic Calabi--Yau variety (together with its base) with another K-trivial variety.
If the base $X$ is not rationally connected, then we can consider different cases:
\begin{itemize}
\item 
if $X$ is not uniruled, then the canonical bundle formula, which we introduce in~\S~\ref{sect.cbf}, implies that the elliptic fibration
$Y \to X$
is isotrivial, see, for example,~\cite{DCS}*{Theorem 2.22};
\item 
if $X$ is uniruled but not rationally connected, then considering its MRC fibration $X \dashrightarrow Z$, to a non-uniruled variety $Z$, one can adapt the above reasoning to show that the induced rational fibration $Y \dashrightarrow Z$ is also isotrivial.
\end{itemize}
When $X$ is rationally connected, a slightly weaker result than the one in Theorem~\ref{ell.cy.thm.gen} was proven by the second and third named authors,~\cite{DCS}, for dimension up to $5$; they showed that the conclusion of the theorem holds for those elliptic Calabi--Yau manifolds with a section under some extra assumptions on the birational structure of the base $X$ of the fibration.
Hence, already in dimension $4$ and $5$, Theorem~\ref{ell.cy.thm.gen} provides a considerable improvement of our current knowledge.
What is worth noting is that, even when $X$ is rationally connected, it is possible that $X$ is birational to a product of a klt log Calabi--Yau pair $(X_1, B_1)$ and of a rationally connected K-trivial variety $X_2$, in which case $X_2$ is forced to have strictly klt singularities.
A rationally connected variety admitting this type of decomposition is said to be of product type, see~\S~\ref{rc.lcy.sect}.
Then, the same reasoning as above, regarding the isotriviality of the rational fibration $Y \dashrightarrow X_2$ can still be applied, see~\S~\ref{rc.cy.sect}. This is definitely the most delicate case to treat in proving the boundedness of elliptic fibrations.

Let us recall that Koll\'ar proved 
in~\cite{koll.elliptic}
that the existence of an elliptic fibration
$Y \to X$,
with 
$Y$ 
a Calabi--Yau manifold, 
satisfying a certain positivity condition with respect to the second Chern class of 
$Y$ 
is a deformation invariant of the any small deformation of $Y$.
Whether the same property holds also in the more singular case considered in the statement of
Theorem~\ref{ell.cy.thm.gen} is definitely a question that is worth studying and carries significant value not only in algebraic geometry but also for string theory (also for the case of an elliptic fibration that does not admit a rational section).

\subsection*{Boundedness for rationally connected log Calabi--Yau pairs}

The heart of the proof of Theorem~\ref{ell.cy.thm} relies on showing that the bases of elliptic Calabi--Yau manifolds form in turn a bounded family, up to isomorphism in codimension one. 
In the elliptically fibred setting, the canonical bundle formula implies that the base $X$ carries a structure of log Calabi-Yau pair, that is, there exists an effective divisor $B$ on $X$ with coefficients in $(0,1)$, $K_X+B$ is numerically trivial and the singularities of the pair $(X, B)$ are mild (klt), see~\S~\ref{pair.sings.sect} for the type of singularities that are allowed. 
The study of the bases of elliptic fibrations is our motivation to study the boundedness of rationally connected log Calabi--Yau pairs. 
We prove the following general result that holds in any dimension.

\begin{theorem}
\label{delta.lcy.thm}
Fix positive integers $d, l$.
Then the set of pairs $(X, B)$ such that
\begin{enumerate}
    \item $(X, B)$ is a projective klt pair of dimension $d$,
    \item $l(K_X+B) \sim 0$, and
    \item $X$ is rationally connected
\end{enumerate}
is log bounded up to flops.
\end{theorem}

This result is not only central to the study of boundedness for elliptic Calabi--Yau varieties, but it maintains its own independent relevance: in view of the unavoidable appearance of singularities in the birational classification of varieties, it is often necessary to work not just with mildly singular varieties, but rather with pairs of a normal projective variety and an effective divisor with coefficients in $[0, 1]$.
All conjectures and questions for Calabi--Yau manifolds discussed so far can be also formulated in the more general context of log Calabi--Yau pairs.
Hence, the importance of boundedness for log Calabi--Yau pairs should be immediately clear, being the natural extension to the realm of pairs of the K-trivial case. As such, then, it is an even more intriguing and complicated problem.
After the completion of this paper (that first appeared in 2020), a similar result was shown also by Jiang and Han,~\cite{HJ22}.

In proving Theorem~\ref{delta.lcy.thm}, we first show that, up to birational modifications, we can decompose $X$ into a tower of fibrations whose general fibres are Fano varieties with bounded singularities.
This result allows us to control the torsion index of the successive log Calabi--Yau pairs that we can construct on the base of each of the fibrations that decompose $X$.
At this point, we show inductively, using Birkar's recent results on boundedness for log Fano fibrations,~\cites{1811.10709, 1811.107092}, that, at each step in this tower of fibrations, if we assume that the base of the fibration belongs to a bounded family, then the same holds also for the total space of the fibration.

The condition that 
$l(K_X+B)\sim 0$
for some fixed 
$l$, 
depending only on the dimension of 
$X$ 
is a conjectured to hold by the so-called Index Conjecture, 
see, for example, \cite{Xu.thesis}.
Clearly, the condition that 
$l(K_X+B)\sim 0$
automatically implies the coefficients of 
$B$ 
vary in the set 
$\{\frac 1l, \frac 2l, \dots, \frac{l-2}{l}, \frac{l-1}{l}\}$.
Assuming the Index Conjecture, it is then not hard to see, thanks to the results of 
\cite{MR3224718},
that Theorem~\ref{delta.lcy.thm} holds under the weaker assumption that the coefficients of $B$ belong to a subset of 
$(0, 1)$
satisfying the descending chain condition\footnote{A subset of $\mathbb R$ satisfies the descending chain condition if any non-increasing sequence is eventually constant.}.

We regard Theorem~\ref{delta.lcy.thm} as providing strong evidence for a more general conjecture, generalizing the BAB Conjecture: the conjecture predicts the boundedness of all rationally connected varieties of log Calabi--Yau type with bounded singularities. 

\begin{conjecture}
\label{eps.lcy.conj}
(cf.~\cite{MP}*{Conj. 3.9},~\cite{CDHJS}*{Conj. 1.3})
Fix a positive integer $d$ and positive real number $\epsilon$.
Then the set of varieties $X$ such that
\begin{enumerate}
    \item $X$ is normal projective and rationally connected of dimension $d$ 
    \item $(X, B)$ is $\epsilon$-lc for some effective $\mathbb{R}$-divisor $B$, and 
    \item $-(K_X+B)$ is nef,
\end{enumerate}
is bounded.
\end{conjecture}

The 2-dimensional case of this conjecture was proved by Alexeev in~\cite{MR1298994}.
In dimension three, combining our techniques with the recent result of Jiang~\cite{1904.09642}, we are able to prove a slightly weaker version of the above conjecture, showing that rather than boundedness, we can obtain boundedness up to isomorphism in codimension one.

\begin{theorem}
\label{eps.lcy.3.thm}
Fix a positive real number $\epsilon$.
The set of varieties $X$ such that
\begin{enumerate}
    \item $X$ is normal projective of dimension $3$,
    \item $(X, B)$ is $\epsilon$-lc for some effective $\mathbb{R}$-divisor $B$,
    \item $-(K_X+B)$ is nef, and 
    \item $X$ is rationally connected
\end{enumerate}
is bounded up to flops.
\end{theorem}

After the completion of this paper, a different version of a boundedness result for (klt non-canonical) log Calabi--Yau pairs in dimension 3 appeared in~\cite{2202.05287}*{Theorem~1.9}.
Afterwards, Birkar provided a full proof to Conjecture~\ref{eps.lcy.conj},~\cite{birkar.last}.

\subsection*{Boundedness for proper Calabi--Yau varieties}
The results presented in the introduction can be extended also to proper varieties:
namely, 
Theorem~\ref{ell.cy.thm} 
and
Corollary~\ref{hodge} 
hold for elliptic proper Calabi--Yau manifolds admitting a rational section;
Theorem~\ref{ell.cy.thm.gen}
holds for proper klt varieties $Y$ that satisfy conditions (2)-(4) from the statement of the theorem;
Theorems~\ref{delta.lcy.thm} 
(resp.~\ref{eps.lcy.3.thm}) 
holds for proper klt pairs 
$(X, B)$ 
that satisfy conditions (2)-(3) 
(resp. (2)-(4)) 
from the statement of the theorem.
Let us observe that in the case of proper varieties by ``bounded up to flops'' we mean that the birational equivalence 
can be decomposed into a finite sequence of rational maps that contract/extract rational curves that have intersection 
$0$
with the canonical divisor 
(resp. the divisor $K_X+B$ in the log Calabi--Yau case).

The possibility of extending the above result to the category of proper varieties follows by the simple observation that, by Hironaka's resolution of singularties and Chow's lemma imply, via the Minimal Model Program, every proper Calabi--Yau manifold (resp. klt log Calabi-Yau pair, klt variety with canonical bundle of $l$-torsion) is isomorphic in codimension one to a projective Calabi--Yau variety with only terminal singularities (resp. klt log Calabi-Yau pair, klt variety with canonical bundle of $l$-torsion), 
cf. the argument in the proof of Theorem~\ref{ell.cy.bound.base.thm}. 
It is then easy to reduce the more general statement for proper varieties to the projective case.

\subsection*{Strategy of proof of Theorem~\ref{ell.cy.thm}-\ref{ell.cy.thm.gen}}

The standard approach to bound a set of fibred varieties is to first bound the set of bases and general fibres, and then lift boundedness from the bases to the set of total spaces. 
Here we follow this very same strategy; the existence of a rational section makes bounding the set of bases of an elliptic fibration the most important and technical step in our proof. 
One of the main novelties here is that we are able to deal with the case where $X$ is rationally connected and $K_X \equiv 0$, an important missing piece in~\cite{DCS}. 
Recall that examples of this type of fibration are known to exists already in dimension $3$, see~\cite{MR1228584}.

From the viewpoint of boundedness, the current techniques in the Minimal Model Program do not provide enough tools to approach rationally connected varieties with $K_X \equiv 0$. 
As predicted in Conjecture~\ref{eps.lcy.conj} for the case when $B=0$, these varieties are expected to be bounded once their singularities are. 
In general, boundedness is expected to hold once the torsion index of the canonical divisor is uniformly bounded, cf.~\cites{CDHJS, 1904.09642}; unfortunately, it is not easy to show that such an hypothesis is satisfied in higher dimension. 
We manage to show that those rationally connected K-trivial varieties that are bases of K-trivial elliptic fibrations with bounded torsion index satisfy boundedness of the torsion index in turn. 
Once uniform boundedness of the torsion index is settled, we can reduce the theorem to boundedness of Fano fibrations (or towers of such fibrations) which is one of the main results of~\cites{1811.10709, 1811.107092}. 
Here for simplicity, we state only the case where the total space is smooth.

\begin{theorem}[cf. Theorem~\ref{index.fibr.thm}]
Fix a positive integer $d$.
Then there exists a positive integer $m = m(d)$ such that if
\begin{enumerate}
    \item $Y$ is a Calabi--Yau manifold of dimension $d$ and
    \item $Y \to X$ is an elliptic fibration with $K_X \equiv 0$,
\end{enumerate}
then $m K_X \sim 0$.
\end{theorem}

Once boundedness of the bases is settled, the following step consists in showing that boundedness also holds for the total space of an elliptic fibration carrying a rational section, see 
Theorem~\ref{ell.cy.bound.base.thm}.
In this case, the idea is to use the Zariski closure $S$ of the rational section $X \dashrightarrow Y$ together with the pullback of a suitable very ample divisor $H$ from the base $X$ with bounded volume -- a bound guaranteed by boundedness of $X$ -- to arrive to a suitable birational model $Y'$ of $Y$ on which $K_Y \equiv 0$, and the strict transform of $S+H$ on $Y'$ has bounded singularities and bounded positive volume.
Then by the results of~\cite{MR3224718} and~\cite{1811.10709}, these conditions imply that 
$Y'$ 
is bounded and the divisors contracted in the birational map 
$Y \dashrightarrow Y'$ 
can be extracted to yield a bounded model 
$Y_1$ 
isomorphic in codimension one to 
$Y$.
These techniques have been recently applied also to the study of the boundedness of 
$d$-dimensional 
minimal models of Kodaira dimension 
$d-1$, 
see~\cites{2005.04254, 2005.05508}, and fibrations on K-trivial varieties of relative dimensione strictly greater than one, see~\cites{2202.07238, 2204.10456}.
Using the ideas of the latter paper and the techniques of this manuscript, it is possible to extend the boundedness up to flops to elliptic Calabi-Yau varieties that admit a section of fixed degree, see~\cite{2005.05508}.
Moreover, the recent work~\cite{2112.01352} shows that boundedness holds for elliptic Calabi--Yau threefolds and that, for elliptic Calabi--Yau varieties in any fixed dimension, it is possible to derive their boundedness from their boundedness up to flops, once the boundedness of the bases of such fibrations is known.
This last result is obtained by resolving a weak version of the so-called Kawamata--Morrison Cone conjecture for elliptic fibrations, see~\cite{2112.01352} for more about this problem and its history.

\subsection*{Acknowledgments} 
GD would like to thank J\'anos Koll\'ar for many valuable discussions.
RS would like to thank Paolo Cascini, Stefano Filipazzi, and Enrica Floris for many useful discussions.
The authors wish to thank Yanning Xu and Stefano Filipazzi for reading preliminary drafts of this work.
The authors wish to heartily thank the anonymous referees for suggesting many changes and improvements to the manuscript.

\section{Preliminaries}
We adopt the standard notation and definitions from~\cite{KMM} and~\cite{KM}, and we freely use those.

A set $I \subset \bR$ is said to be a DCC (resp. ACC) set, if it does not contain any strictly decreasing (resp. increasing) sequence $\{i_k\}$ of elements of $I$.

\subsection{Pairs and singularities}\label{pair.sings.sect}
A {\it log pair} $(X, B)$ consists of a normal projective variety $X$ and an effective $\bR$-divisor $B$ on $X$ such that $K_X+B$ is $\bR$-Cartier.
\\
Given $f\colon Y\rightarrow X$ a log resolution of the log pair $(X, B)$, we write
\[
K_Y +B' =f^*(K_X+B),
\]
where $B'$ is the sum of the strict transform $f^{-1}_\ast B$ of $B$ on $Y$ and a divisor completely supported on the exceptional locus of $f, \; B'= f^{-1}_\ast B + \sum a_i E_i$.  
We will denote by $\mu_D B'$ the multiplicity of $B'$ along a prime divisor $D$ on $Y$.
For a non-negative real number $\epsilon$, the log pair $(X,B)$ is called
\begin{itemize}
\item[(a)] \emph{$\epsilon$-kawamata log terminal} (\emph{$\epsilon$-klt}, in short) if $\mu_D B'< 1-\epsilon$ for all $D \subset Y$;
\item[(b)] \emph{$\epsilon$-log canonical} (\emph{$\epsilon$-lc}, in short) if $\mu_D B'\leq  1-\epsilon$ for all $D \subset Y$;
\item[(c)] \emph{terminal} if  $\mu_D B'< 0$ for all $f$-exceptional $D \subset Y$ and all possible choices of $f$;
\item[(d)] \emph{canonical} if  $\mu_D B'\leq 0$ for all $f$-exceptional $D \subset Y$ and all possible choices of $f$.
\end{itemize}

Let us note that $0$-klt (resp., $0$-lc) is just klt (resp., lc) in the usual sense. 
We say that a log pair $(X, B)$ has {\it strictly klt} singularities if $(X, B)$ is $0$-klt but not $1$-lc.
When we consider a log pair $(X, B)$ with $B=0$, we will usually write $X$ instead of $(X,0)$.

The {\it log discrepancy} of a prime divisor $D$ on $Y$ is defined to be $a(D, X, B):=1-\mu_D B'$.
It does not depend on the choice of the log resolution $f$.

\subsection{Generalised pairs}
\label{gen.pair.sect} 
For the definition of b-divisor and related notions, we refer the reader to~\cite{bir-zh}.
There, the authors introduced also the notion of generalised pairs. 
Let us recall that a b-$\mathbb{R}$-divisor $\mathbf{N}$ is said to descend to the divisor $N'$ on a model $X'$ if $\mathbf{N}$ equals the Cartier closure of its trace $\mathbf{N}_{X'}$ on $X'$ and $\mathbf{N}_{X'}=N'$.
\begin{definition}
Let $Z$ be a variety. 
A generalised polarised pair over $Z$ is a tuple $(X' \to X, B, M')$ consisting of the following data:
\begin{itemize}
\item a normal variety $X \to Z$ projective over $Z$ equipped with a projective birational morphism $\phi:X' \rightarrow X$, 
\item an effective $\bR$-Weil divisor $B$ on $X$,
\item a b-$\bR$-Cartier b-divisor $\mathbf{M}$ over $X$ which descends on $X'$ such that $M':=\mathbf{M}_{X'}$ is nef over $Z$, and
\item $K_X +B+ M$ is $\bR$-Cartier, where $M := \phi_\ast M'$.
\end{itemize}
\end{definition}

To streamline notation, we will often indicate a generalised pair simply by 
$(X, B+M)$
rather than by 
$(X' \rightarrow X, B, M')$, 
where we adopt the notation of the previous definition.
We call $M'$ the nef part of the generalised pair.

Similarly to log pairs, we can define discrepancies and singularities for generalised pairs. 
Replacing $X'$ with a higher birational model, we can assume that $\phi$ is a log resolution of $(X, B)$. 
Then we can write 
\[
K_{X'} + B'+M'=\phi^*(K_X+B+M)
\]
for some uniquely determined divisor $B'$. 
For a prime divisor $D$ on $X'$ the generalised log discrepancy $a(D,X,B + M)$ is defined to be $1-\mu_D B'$.
We say $(X, B+M)$ is generalised lc (resp. generalised klt, generalised $\epsilon$-lc) if for each $D$ the generalised log discrepancy $a(D,X,B+M)$ is $\geq0$ (resp.$> 0, \; \geq \epsilon$).

\subsection{Minimal model program}
Moreover, we will make use of the Minimal Model Program (MMP, in short) for (generalised) pairs with non-pseudo-effective log canonical class. 
In this case, the existence and termination of the MMP has been established in the pair setting by Birkar-Cascini-Hacon-M$^\text{c}$Kernan,~\cite{BCHM}*{Cor. 1.3.3}, and by Birkar-Zhang in the generalised setting,~\cite{bir-zh}*{Lemma 4.4}. 

\begin{theorem}
\label{mmp}
\cites{bir-zh, BCHM}
Let $(X, B+M)$ be a projective $\mathbb{Q}$-factorial generalised klt pair. 
Assume $K_{X}+B+M$ is not pseudo-effective. 
Then we may run a $(K_{X}+B+M)$-MMP $g \colon X \dashrightarrow Y$ that terminates with a Mori fibre space $f\colon Y \rightarrow T$.
\end{theorem}

Let us recall the definition of Mori fibre space. 

\begin{definition}
Let $(X, B+M)$ be a generalised klt pair and let $f \colon X\rightarrow T$ be a projective morphism of normal varieties with $\dim(T)<\dim(X)$. 
Assume that $f_\ast \mathcal{O}_X=\mathcal{O}_T$. 
Then $f$ is a Mori fibre space if 
	\begin{enumerate}
		\item $X$ is $\mathbb{Q}$-factorial,
		\item $f$ is a primitive contraction, i.e. the relative Picard number $\rho(X/T)=1$ and 
		\item $-(K_{X}+B+M)$ is $f$-ample.
	\end{enumerate}
\end{definition}

\subsection{Boundedness of pairs} 
\label{bound.sect}

We recall the different notions of boundedness for varieties and log pairs.
\begin{definition}
A collection of projective varieties $\mathfrak{D}$ is said to be \emph{bounded} (resp.,  \emph{birationally bounded}, or \emph{bounded in codimension one}) if there exists  $h\colon \mathcal{Z}\rightarrow S$ a projective morphism of schemes of finite type such that each $X\in \mathfrak{D}$ is isomorphic (resp., birational, or isomorphic in codimension one) to $\mathcal{Z}_s$ for some closed point $s\in S$.
\end{definition}

\begin{definition}
\label{def.bound.pairs}

A collection of projective log pairs 
$\mathfrak{D}$ 
is said to be \emph{log birationally bounded} (resp.,  \emph{log bounded}, or \emph{log bounded in codimension one}) if there is a quasi-projective variety 
$\mathcal{Z}$, 
which may possibly be reducible, a reduced divisor 
$\mathcal{E}$ 
on 
$\mathcal Z$, 
and a projective morphism 
$h\colon \mathcal{Z}\to S$, 
where
$S$ 
is of finite type and 
$\mathcal{E}$ 
does not contain any fibre, such that for every 
$(X,B)\in \mathfrak{D}$, 
there is a closed point 
$s \in S$ 
and a birational map 
$f \colon \mathcal{Z}_s \dashrightarrow X$ 
(resp., isomorphic, or isomorphic in codimension one) such that 
$\mathcal{E}_s$ 
contains the support of 
$f_*^{-1}B$ 
and any 
$f$-exceptional 
divisor (resp., 
$\mathcal{E}_s$ 
coincides with the support of 
$f_*^{-1}B$, $\mathcal{E}_s$ 
coincides with the support of 
$f_*^{-1}B$).
\newline
A collection of projective log pairs 
$\mathfrak{D}$ 
is said to be \emph{strongly log bounded} if there is a quasi-projective log pair 
$(\mathcal{Z}, \mathcal B)$ 
and a projective morphism 
$h\colon \mathcal{Z}\to S$, 
where 
$S$ 
is of finite type, such that 
$\mathcal{B}$ 
does not contain any fibre of 
$h$, 
and for every 
$(X,B)\in \mathfrak{D}$, 
there is a closed point 
$s \in S$ 
and an isomorphism 
$f \colon \mathcal{Z}_s \to X$ 
such that 
$f_\ast \mathcal{B}_s= B$.
\end{definition}

In the case of Calabi--Yau pairs, we will use an equivalent notion of log boundedness in codimension one that is more suitable for our proofs.

\begin{remark}
\label{def:bound.mod.flops}
We use the notation of Definition~\ref{def.bound.pairs}.
Given a collection 
$\mathfrak{D}$ 
projective klt Calabi--Yau varieties 
(resp., klt log Calabi--Yau pairs) 
which is bounded in codimension one 
(resp., log bounded in codimension one), 
then by replacing 
$\mathcal Z$ 
with its normalization and decomposing 
$S$
into a a finite union of locally closed 
(in the Zariski topology) 
subsets, we can assume that 
$S$
is smooth and that the fibre 
$\mathcal{Z}_{s}$ 
isomorphic in codimension one to a given
$X \in \mathfrak D$
(resp.
$(X, B) \in \mathfrak D$) 
is normal projective, and 
$K_{\mathcal{Z}_s}$ 
is 
$\bQ$-Cartier 
(resp., 
$K_{\mathcal{Z}_s}+f_*^{-1}B$ 
is 
$\bR$-Cartier). 
When such situation is realized, we will then say that the collection
$\mathfrak D$
is {\it bounded modulo flops} 
(resp., {\it log bounded modulo flops}). 
The indication ``modulo flops" comes from the fact that, if we assume that 
$X$ 
and 
$\mathcal{Z}_s$ 
are both 
$\bQ$-factorial, 
then, under the above assumptions, they are connected by a finite sequence of 
$K_X$-flops 
(resp. 
$K_X+B$-flops): 
indeed, this is achieved by running a 
$(K_X+\delta f_\ast H)$-MMP
(resp. 
$(K_X+B+\delta f_\ast H)$-MMP), 
where 
$H$ 
is an ample divisor on 
$\mathcal{Z}_s$, 
and 
$\delta$ 
is a sufficiently small positive number, 
see~\cites{BCHM, MR2426353}.
\end{remark}

If $\mathfrak{D}$ is a collection of klt log Calabi--Yau pairs which is log bounded modulo flops, $(X, B)\in \mathfrak{D}$, and $f \colon \mathcal{Z}_s \dashrightarrow X$ is an isomorphism in codimension one as in the definition above, then $(\mathcal{Z}_s, f_*^{-1}B)$ is again a klt log Calabi--Yau pair by the Negativity Lemma.  
Moreover,  $(X, B)$ is $\epsilon$-lc if and only if $(\mathcal{Z}_s, f_*^{-1}B)$ is so. 
The same statement holds for a set $\mathfrak{D}$ of klt Calabi--Yau varieties which is bounded modulo flops.

To show that a given set of pairs is log birationally bounded, we will mainly use the following theorem which is a combination of results in~\cites{MR3034294, MR3224718}.

\begin{theorem} 
\cite{MR3034294}*{Theorem 3.1}, 
\cite{MR3224718}*{Theorem 1.3} 
\label{main.hmx}
Fix two positive integers $d$ and $V$ and a DCC set $I\subset [0,1]$.
Then the collection $\mathfrak D$ of log pairs $(X,B)$ satisfying
	\begin{enumerate}
		\item $X$ is a projective variety of dimension $d$,
		\item $(X, B)$ is lc, 
		\item the coefficients of $B$ belong to $I$, and
		\item $0<\Vol(K_X+B)\leq V$,
	\end{enumerate}
	is log birationally bounded.
\end{theorem}

In certain special cases it is possible to deduce boundedness from log birational boundedness. 
\begin{theorem}
\cite{MR3224718}*{Theorem 1.6}
\label{hmx_1.6}
	Fix a positive integer $d$ and two positive real numbers $b$ and $\epsilon$. 
	Let $\mathfrak{D}$ be a collection of log pairs $(X,B)$ 
	such that: 
	\begin{enumerate}
		\item $X$ is a projective variety of dimension $d$,
		\item $K_{X}+B$ is ample,  
		\item the coefficients of $B$ are at least $b$, and
		\item the log discrepancy of $(X,B)$ is greater than $\epsilon$.
	\end{enumerate} 
	If $\mathfrak{D}$ is log birationally bounded then $\mathfrak{D}$ is a log bounded set of log pairs.
\end{theorem}

Theorem~\ref{hmx_1.6} can be strengthened when we further impose control on the singularities and coefficients of pairs.
The following result is a straightforward consequence of~\cite{1610.08932}*{Theorem 6} and~\cite{1804.10971}*{Theorem 1.3}.

\begin{theorem}
\label{mst+fil.thm}
Fix a natural number $d$, a positive rational number $C$, a positive real number $\epsilon$, and $I \subset \mathbb{Q} \cap [0, 1)$ a finite set.
The collection $\mathfrak{D}$ of all pairs $(X, B)$ such that
\begin{enumerate}
    \item 
$(X, B)$ is $\epsilon$-klt, 
projective, of dimension 
$d$,
    
    \item
the coefficients of 
$B$ belong to 
$I$,
    
    \item
$K_X+B$ 
is big and nef, and  
    
    \item
$\Vol(K_X+B) \leq C$
\end{enumerate}
is strongly log bounded.
\end{theorem}

In the course of our treatment, we will need to use several times the following technical result which allows us to pass from bounded collections of log pairs to log bounded ones.

\begin{proposition}
\label{fam.good.prop}
Fix a finite set 
$\mathfrak{R} \subset (0, 1)$ 
and a positive real number 
$\epsilon$.
Let 
$\mathfrak{D}$ 
be a bounded collection of 
$\epsilon$-lc 
pairs.
Assume that for any pair 
$(X, B) \in \mathfrak{D}$ 
the coefficients of 
$B$ 
belong to 
$\mathfrak{R}$. 
Then 
$\mathfrak{D}$ 
is strongly log bounded.
\end{proposition}

The important point in the statement of the proposition is that there exists a proper morphism of quasi-projective varieties $h \colon \mathcal{Z} \to S$ with $S$ smooth, and a divisor $\mathcal{B}$ on $\mathcal{Z}$ such that for any $(X, B) \in \mathfrak{D}$ there exists $s \in S$ and an isomorphism $f_s \colon \mathcal Z_s \to X$ with $f_s^\ast(B) =\mathcal{B}_s$.
In this context, both $\mathcal Z$ and $S$ are not necessarily irreducible, but, nonetheless, they are of finite type.

\begin{proof}
By definition of boundedness, see Definition~\ref{def.bound.pairs}, there exists a couple 
$(\mathcal Z, \mathcal{E})$ 
with 
$\mathcal E$ 
reduced divisor, and a projective morphism of quasi-projective varieties 
$h \colon \mathcal Z \to S$ 
such that for any pair 
$(X, B) \in \mathfrak{D}$ 
there exists 
$s \in S$ 
and an isomorphism 
$f_s \colon \mathcal{Z}_s \to X$ 
such that 
$f_s$ 
maps 
$\mathcal{E}_s$ 
to the support of 
$B$.
\newline
Decomposing 
$S$ 
into a finite union of locally closed subsets and possibly discarding some components, we may assume that:
\begin{enumerate}
    \item 
$S$ is smooth,
    \item 
$\mathcal E$ 
does not contain any fibre of 
$h$ and it is flat over 
$S$, 
and 
    \item
every fibre 
$\mathcal{Z}_s$ 
is a normal variety.
\end{enumerate}
Further decomposing 
$S$ 
into a finite union of locally closed subsets, we may assume that the assumptions (1-3) above are satisfied and moreover that 
\begin{enumerate}
\setcounter{enumi}{3}
    \item 
for any 
$s \in S$ 
and any irreducible component 
$\mathcal E'$ 
of 
$\mathcal E$, 
$\mathcal E'\vert_{\mathcal Z_s}$ 
is either empty or an irreducible prime divisor on 
$\mathcal Z_s$.
\end{enumerate}
Taking the decomposition
$\mathcal E = \sum_{i =1}^l \mathcal E_i$
into prime components, 
then for any 
$l$-tuple 
\begin{align*}
\overline a = (a_1, \dots, a_l) \in \mathfrak R^l:= \underbrace{\mathfrak R \times \dots \times \mathfrak R}_{\text{$l$ times}}
\end{align*}
we define 
$\mathcal E_{\overline a} :=
\sum_{i =1}^l a_i \mathcal E_i$.
Hence, 
by the assumptions in the statement of the proposition and assumption (4) above, 
for any 
$(X, D) \in \mathcal D$ 
there exist
$s \in S$, 
$\overline a \in \mathfrak R^l$,
and an isomorphism 
$f_s \colon \mathcal{Z}_s \to X$ 
such that 
$f_s^\ast B = \mathcal E_{\overline a} \vert_{\mathcal Z_s}$.
Hence, defining
\begin{align*}
(\tilde{\mathcal Z}, \tilde{\mathcal E}) 
:=
\bigsqcup_{\overline a \in \mathfrak R^l}    
(\mathcal Z, \mathcal E_{\overline a}), 
\qquad 
\tilde{S} 
:=
\bigsqcup_{\overline a \in \mathfrak R^l}  
S,
\end{align*}
by construction, there exists a natural morphism
$\tilde h \colon \tilde{ \mathcal Z} \rightarrow \tilde S$, 
and for any 
$(X, D) \in \mathcal D$ 
there exist 
$\tilde s \in \tilde S$ 
and an isomorphism 
$\tilde {f}_{\tilde s} \colon \tilde{\mathcal{Z}}_{\tilde s} \to X$ 
such that 
$\tilde {f}_{\hat s}^\ast B= \tilde{\mathcal E}\vert_{\tilde{\mathcal Z}_{\tilde s}}$
which proves the desired claim, again by Definition~\ref{def.bound.pairs}.
\end{proof}

\subsection{Boundedness of Fano fibrations}

To simplify the statements in this section, we recall the following definition introduced in~\cite{1811.107092}. 

\begin{definition}
\cite{1811.107092}*{Definitions~1.1 and~2.1} 
\label{fano-fibr.def}
Let 
$d,r$ 
be natural numbers and 
$\epsilon$ 
be a positive real number. 
\begin{enumerate}
	\item 
A 
\emph{$(d,r,\epsilon)$-Fano 
type (log Calabi-Yau) fibration} consists of a log  pair 
$(X,B)$ 
and a contraction 
$f\colon X\to Z$ 
such that we have the following:
\begin{itemize}
	\item 
$(X,B)$ 
is a projective 
$\epsilon$-lc 
pair of dimension 
$d$,

	\item 
$K_X+B\sim_\mathbb{R} f^\ast L$ 
for some 
$\mathbb{R}$-divisor 
$L$, 

\item $-K_X$ is big over  $Z$, i.e. $X$ is of Fano type over $Z$,

\item $A$ is a very ample Cartier divisor on $Z$ with $A^{\dim Z}\le r$, and 

\item $A-L$ is ample.
\end{itemize}
\item 
A generalised
\emph{$(d,r,\epsilon)$-Fano 
type (log Calabi-Yau) fibration} consists of a generalised log  pair 
$(X,B+m)$ 
and a contraction 
$f\colon X\to Z$ 
such that we have the following:
\begin{itemize}
	\item 
$(X,B+M)$ 
is a projective 
$\epsilon$-lc 
pair of dimension 
$d$,

	\item 
$K_X+B+M\sim_\mathbb{R} f^\ast L$ 
for some 
$\mathbb{R}$-divisor 
$L$, 

	\item 
$-K_X$ 
is big over
$Z$, 
i.e., 
$X$ 
is of Fano type over 
$Z$,

	\item 
$A$ 
is a very ample Cartier divisor on 
$Z$ 
with 
$A^{\dim Z}\le r$, 
and 

	\item 
$A-L$ 
is ample.
\end{itemize}
\end{enumerate}
\end{definition}

Birkar has shown that 
$(d,r,\epsilon)$-Fano 
type fibrations and their generalised analogues are bounded. 
Let us recall the following result which will be relevant later in our treatment.

\begin{theorem}
\cite{1811.107092}*{Theorem 2.3}
\label{bir.bound.fibr.thm}
Fix positive natural numbers $d,r$ 
and positive real numbers 
$\epsilon, \tau$. 
Let 
$\mathfrak{D}$ 
be the set of pairs
$(X, \Delta)$ 
such that 
\begin{itemize}
\item 
there exists a generalised 
$(d,r,\epsilon)$-Fano 
type fibration
$(X, B+M) \to Z$; and, 
\item 
$0 \leq \Delta \leq B$ 
and the coefficients of 
$\Delta$ 
are 
$\geq \tau$.
\end{itemize}
Then $\mathfrak D$ is log bounded.
\end{theorem}

The following corollary is a direct consequence of Theorem~\ref{bir.bound.fibr.thm}.

\begin{corollary}
\label{bir.bound.fibr.cor}
Fix positive real numbers 
$\epsilon$, 
$\tau$.
Let 
$\mathfrak D$ 
be a bounded set of varieties.
Let 
$\mathfrak D'$ 
be the set of pairs
$(X', \Delta')$ 
such that 
\begin{enumerate}
\item 
there exist 
$X \in \mathfrak D$ 
and a generalised log pair structure
$(X, B+M)$
on 
$X$
which 
is $\epsilon$-generalized lc and 
$K_X+B+M \sim_\mathbb{R} 0$,

\item 
there exists a birational morphism 
$\pi \colon X' \to X$
and a generalised log pair structure
$(X', B'+M')$
on
$X'$
such that
$\pi$
is crepant for the generalised log structures 
on
$X$
and
$X'$, 
that is, 
$K_{X'}+B'+M'=\pi^\ast(K_X+B+M)$, and,

\item 
$0 \leq \Delta' \leq B'$ 
and the coefficients of 
$\Delta'$ 
are 
$\geq \tau$.
\end{enumerate}
Then 
$\mathfrak D'$ 
is log bounded.
\\
Moreover, if in the above statement we substitute assumption (2)
with the weaker assumption
\begin{enumerate}
\item[(2')]
there exists a birational contraction
$\pi \colon X' \dashrightarrow X$
and a generalised log pair structure
$(X', B'+M')$
on
$X'$
such 
$\pi$
is crepant for the generalised log structures 
on
$X$
and
$X'$, 
that is, 
$K_{X'}+B'+M'=\pi^\ast(K_X+B+M)$,
\end{enumerate}
then
$\mathfrak D'$ 
is log bounded in codimension one.
\end{corollary}

Let us observe that, in the above statement,
$\pi^\ast(K_{X}+B+M)$ 
is well defined even when 
$\pi$ 
is not a morphism, but merely a birational map, since we are assuming that 
$\pi$ 
is a birational contraction, that is, 
$\pi$ 
does not contract any codimension one subvariety of 
$X'$. 

\begin{remark}
\label{rmk.q-factor.bound}
In particular, if
$\mathfrak D$
is a bounded (resp. bounded in codimension one) set of klt varieties that all admit a klt generalized log Calabi--Yau structure,
then Corollary~\ref{bir.bound.fibr.cor} shows that the set of 
all $\mathbb Q$-factorializations 
of varieties in 
$\mathfrak D$
is bounded (resp. bounded in codimension one).
\end{remark}

\begin{proof}
Since 
$\mathfrak D$ 
is bounded, 
we can assume that all the varieties in 
$\mathfrak D$
have the same dimension 
$d$.
Moreover, again by boundedness of
$\mathfrak D$,
there exists a positive integer 
$C=C(\mathfrak D)$
such that for any 
$X \in \mathfrak D$
there exists a very ample Cartier divisor
$H_X$ 
on 
$X$
such that
$H^d \leq C$.
To conclude then, by the observations that we just made and by assumption (1-3) in the statement of the corollary, 
the birational morphism 
$\pi \colon X' \to X$ 
is a
$(\epsilon, d, C)$-Fano 
type fibration: 
indeed, the birationality of $\pi$ implies that 
$-K_{X'}$ 
is big over 
$X$
and assumption (2) shows that 
$K_{X'}+B'+M' \sim_{R} 0$.
Thus, the conlcusion follows at once by applying Theorem~\ref{bir.bound.fibr.thm}.
\\
To prove the second part of the statement, the one depending from assumption 
(2'), 
it suffices to observe that if 
$\pi \colon X' \dashrightarrow X$
is a birational contraction which is crepant for the generalized log Calabi--Yau pair 
$(X', B'+M')$, 
then there exists a commutative diagram
\begin{align}
\label{eqn.diag.crep.resol}
    \xymatrix{
X' \ar@{-->}[rr]^{\psi} \ar@{-->}[rd]^{\pi} 
&
&
X'' \ar[dl]_{\xi} \\
&
X
&   }
\end{align}
where 
$\psi$ 
is an isomorphism in codimension one and 
$\xi$
is a birational morphism.
To prove the existence of the diagram in~\eqref{eqn.diag.crep.resol}, it suffices to notice that since $\pi$ is a birational contraction, then it only contracts a finite number of prime divisors 
$\{E_1, \dots, E_l\}$
on 
$X'$.
Each
$E_i$
has generalized log discrepancy in 
$(\epsilon, 1]$
for 
$(X', B'+M')$, 
or, equivalently, 
for 
$(X, B+M)$.
Hence, by 
\cite{bir-zh}*{Lemma~4.6}
there exists a birational morphism
$\xi \colon X'' \to X$ 
and a generalized $\epsilon$-lc pair
$(X'', B''+M'')$ 
such that 
$\xi$
extracts all and only the valuations associated to the 
$E_i$, 
$i=1, \dots, l$,
and 
$K_{X''}+ B''+M'' = \xi^\ast (K_X+B+M)$.
Hence, by construction, 
$X'$ 
and
$X''$
are isomorphic in codimension one, which concludes the proof of the existence of the diagram. 
\\
Hence, 
defining 
$B'' := \psi_\ast B'$,
$M'' := \psi_\ast M'$,
$\Delta'' := \psi_\ast \Delta'$,
then 
$0 
\sim_{\mathbb R}
K_{X''}+B''+M''
\sim_{\mathbb R}
\xi^\ast(K_X+B+M)
$ 
and 
$\psi$ 
is crepant with respect to the generalized log Calabi--Yau structures 
$(X', B'+M')$,
$(X'', B''+M'')$.
By the first part of the statement, it then follows that the set 
$\mathfrak D''$
of pairs
$(X'',  \Delta'')$
just constructed is log bounded, which in turn implies, by Definition~\ref{def.bound.pairs}, that $\mathfrak D'$ is log bounded in codimension one.
\end{proof}

\subsection{Log bounded families of morphisms}
\label{bound.morph.subsect}

We extend the notion of boundedness to collection of log pairs endowed with morphisms.

\begin{definition}
\label{bound.morph.def}
\begin{enumerate}
    \item[(i)] 
A set 
$\mathfrak F$ 
of morphisms is the datum of a collection of 4-tuples 
$(f, Y, B, X)$ 
where 
$f \colon Y \to X$ 
is a surjective morphism of projective varieties and the pair 
$(Y, B)$ 
is a log pair.

    \item[(ii)] 
We say that a set $\mathfrak F$ of morphisms is \emph{log bounded} (resp., \emph{strongly log bounded}) if there exist quasi-projective varieties $\mathcal Y, \mathcal X$, a reduced divisor $\mathcal{E}$ on $\mcy$ (resp., a log pair $(\mathcal Y, \mathcal{E})$), and a commutative diagram of projective morphisms 
\begin{align*}
\xymatrix{
\mathcal Y \ar[rr]^{\phi} \ar[dr] & 
&
\mathcal X \ar[dl]\\
 & S &
}
\end{align*}
such that 
\begin{itemize}
    \item 
$S$ is of finite type, 
    \item 
for any 4-tuple $(f, Y, B, X)\in \mathfrak{F}$, there is a closed point $s \in S$ and isomorphisms $p \colon \mathcal{X}_s \to X$, $q \colon \mathcal{Y}_s \to Y$ such that the following diagram commutes
\begin{align*}
\xymatrix{
\mathcal{Y}_s \ar[r]^{\phi\vert_{\mathcal{Y}_s}} \ar[d]^{q} & 
\mathcal{X}_s \ar[d]^{p}\\
Y  \ar[r]^{f} &
X,
}
\end{align*}
and

\item 
$q_\ast \mathcal{E}_{s}$ coincides with the support of $B$ (resp., $q_\ast \mathcal{E}_{s}=B$).
\end{itemize}
\end{enumerate}
\end{definition}

\begin{remark}
\label{bound.morph.imply.bound.rmk}
Let $\mathfrak{F}$ be a bounded (resp., log bounded) collection of morphisms.
Then the sets
\begin{align*}
    \mathfrak T := & \{(Y, B) \; \vert \; \exists (f, Y, B, X) \in \mathfrak F\},\\
    \mathfrak B := & \{ X \; \vert \; \exists (f, Y, B, X) \in \mathfrak F\}
\end{align*}
are a bounded (resp., log bounded) set of pairs and a bounded set of varieties, respectively.
This is an immediate consequence of the conditions of Definition~\ref{bound.morph.def}.
\end{remark}

\begin{lemma}
\label{bound.morph.lemma}
Let $d$ be a positive integer, $\epsilon$ be a positive real number, and let $\mathcal R \subset \mathbb{Q} \cap (0,1)$ be a finite set.
\newline
Let $\mathcal D$ be a log bounded set of pairs $(Y, B)$ such that
\begin{enumerate}
    \item 
$(Y, B)$ is projective, $\epsilon$-klt of dimension $d$,
    
    \item
the coefficients of $B$ are in $\mathcal{R}$,

    \item 
$K_Y+B$ is semi-ample.
\end{enumerate}
\noindent
Then, the set $\mathfrak F$ of 4-tuples $(f, Y, B, X)$ such that
\begin{enumerate}
    \item[(i)]
there exists a log pair $(Y, B) \in \mathcal D$,

    \item[(ii)]
$X= {\rm Proj}(\oplus_{i=0}^\infty H^0(Y, \mathcal{O}_Y(i(K_Y+B))))$, and

    \item[(iii)] 
$f \colon Y \to X$ is the Iitaka fibration of $K_Y+B$
\end{enumerate}
is a strongly log bounded set of morphisms.
\end{lemma}

\begin{proof}
As 
$\mathcal D$ 
is a log bounded collection of 
$\epsilon$-klt 
pairs and 
$\mathcal R$ 
is finite, we can assume that 
$\mathcal D$ 
is strongly log bounded by Proposition~\ref{fam.good.prop}.
Hence, there exists a scheme of finite type 
$S$, 
a log pair 
$\pi \colon (\mathcal Y, \mathcal B) \to S$ 
over 
$S$ 
such that for any log pair 
$(Y, B) \in \mathcal{D}$, 
there exists 
$s \in S$ 
and 
$(\mathcal{Y}_s, \mathcal{B}_s)$ 
is isomorphic to $(Y, B)$.
\newline
Decomposing $S$ into a finite union of locally closed subsets and possibly discarding some components, we may assume that every fibre $\mathcal{Y}_s$ is a variety and that $\mathcal{B}$ does not contain any fibre. 
Decomposing $S$ into a finite union of locally closed subsets and passing to a finite cover of $S$, we may assume that for any $s\in S$ there is a 1-1 correspondence between the irreducible components of $\mathcal B$ and those of $\mathcal B_s$.
Decomposing $S$ into a finite union of locally closed subsets, we may assume that $S$ is smooth.
By~\cite{Bir16a}*{Lemma 2.25}, we can assume that there exists $I=I(\mathfrak D)$ such that $I(K_\mathcal{Y}+\mathcal B)$ is Cartier along any fibre $\mathcal{Y}_s$. 
Up to shrinking $S$ and decomposing it into a finite union of locally closed subsets, we can also assume that the set 
\begin{align*}
   S':= \{s\in S \ \vert \ \exists ( Y, B) \in \mathfrak{D} \textrm{ such that } (Y, B) \simeq (\mathcal{Y}_s, \mathcal{B}_s)\}
\end{align*}
is Zariski dense in $S$.
By~\cite{1810.01990}*{Proposition 2.10}, up to shrinking $S$ and decomposing it into a finite union of locally closed subsets, we can assume that $\mathcal{Y}$ is $\mathbb{Q}$-factorial and $S'$ is still dense.
In particular, this implies that for all $s \in S$, $(\mathcal Y_s, \mathcal B_s)$ is $\epsilon$-klt.

\medskip

{\bf Claim 1}. 
{\it 
Decomposing $S$ into a finite union of locally closed subsets, we may assume that, for all $l \in \mathbb N_{>0}$, the restriction map 
\begin{align}
\label{lift.eqn}
    H^0(\mathcal Y, \mathcal{O}_{\mathcal{Y}}(lI(K_{\mathcal{Y}} + \mathcal B))) \to 
    H^0(\mathcal{Y}_s, \mathcal{O}_{\mathcal{Y}_s}(lI(K_{\mathcal{Y}_s} + \mathcal B\vert_{\mathcal{Y}_s}))),
\end{align}
is surjective at any point $s \in S$, and that for any connected component $\bar{S}$ of $S$
\begin{align}
\label{const.rk.eqn}
h^0(\mathcal{Y}_s, \mathcal{O}_{\mathcal{Y}_s}(lI(K_{\mathcal{Y}_s} + \mathcal B\vert_{\mathcal{Y}_s}))) \textrm{ is independent of $s \in \bar{S}$}.
\end{align}
}
\begin{proof}
Decomposing $S$ into a finite union of locally closed subset, we can assume that there exists a log resolution $\psi\colon (\mathcal{Y}', \overline{\mathcal{B}}) \to \mathcal Y$ of $(\mathcal Y, \mathcal B)$, where $\overline{\mathcal{B}}:= \psi^{-1}_\ast \mathcal B + E$ and $E$ is the exceptional divisor of $\psi$, and for any $s \in S$, $(\mathcal{Y}'_s, \overline{\mathcal{B}}_s)$ is a log resolution of $(\mathcal{Y}_s, \mathcal{B}_s)$.
In particular, for any $s \in S$, 
\begin{align*}
H^0(\mathcal{Y}_s, \mathcal{O}_{\mathcal{Y}_s}(l(K_{\mathcal{Y}_s}+\mathcal{B}_s))) = 
H^0(\mathcal{Y}'_s, \mathcal{O}_{\mathcal{Y}'_s}(l(K_{\mathcal{Y}'_s}+\overline{\mathcal{B}}^{\epsilon}_s))), \ 
\overline{\mathcal{B}}^{\epsilon}:= \psi^{-1}_\ast \mathcal B + (1-\epsilon) E,
\end{align*}
where in the equation above we have used that the pairs we consider are $\epsilon$-klt.
Hence, the properties in~\eqref{lift.eqn}-\eqref{const.rk.eqn} follow from~\cite{MR3779687}*{Theorems~1.4,~4.2}, since for any $s \in S'$ the pair $(\mathcal{Y}_s, \mathcal{B}_s)$ has a good minimal model, by assumption (3) in the statement of the Lemma. 
\end{proof}
\noindent
If $s \in S$ is a point such that $\vert mI(K_{\mathcal{Y}_s} + \mathcal B\vert_{\mathcal{Y}_s})\vert $ is base point free, then it follows from Claim 1 that the natural restriction morphism
\begin{align}
\label{eqn:natural.restr.fiber.}
\xymatrix{
    \pi^\ast \pi_\ast \mathcal{O}_{\mathcal{Y}}(mI(K_{\mathcal{Y}} + \mathcal B)) \ar[r] & 
    \mathcal{O}_{\mathcal{Y}}(mI(K_{\mathcal{Y}} + \mathcal B))
    \ar[r] &
    \mathcal{O}_{\mathcal{Y}, y}(mI(K_{\mathcal{Y}} + \mathcal B))
    }
\end{align}
is surjective at every $y \in \mathcal{Y}_s$. 
As having trivial cokernel is an open condition and $\pi$ is proper, then there exists a Zariski neighborhood $U$ of $s$ in $S$ such that for all $s' \in U$,' the surjectivity of the composition of morphisms in~\eqref{eqn:natural.restr.fiber.} holds for all $y \in \mathcal{Y}_{s'}$.
Hence, up to decomposing $S$ into a finite union of locally closed subsets, we can assume that there exists a positive integer $m'$ such that $m'I(K_{\mathcal{Y}} + \mathcal B)$ is relatively base point free over $S$.
The relative Iitaka fibration for $K_{\mathcal{Y}} + \mathcal B$ over $S$
\begin{align*}
    \xymatrix{
    (\mathcal Y , \mathcal{B})\ar[dr] \ar[rr]^\phi &
    & 
    \mathcal X \ar[dl]\\
    & S &
    }
\end{align*}
provides the desired bounded family of triples:
in fact, given $(Y, B) \in \mathfrak D$, if $s \in S$ is such that there exists an isomorphism $h \colon \mathcal{Y}_s \to Y$ and $h_\ast\mathcal{B}_s = B$, then~\eqref{lift.eqn}-\eqref{const.rk.eqn} imply that there exists also an isomorphism $g \colon \mathcal{X}_s \to X$ and a commutative diagram
\begin{align*}
    \xymatrix{
    \mathcal{Y}_s \ar[r]^{\phi\vert_{\mathcal{Y}_s}} \ar[d]^h &
    \mathcal{X}_s \ar[d]^g \\
    Y \ar[r]^{f} &
    X.
    }
\end{align*}
\end{proof}

Let us recall the following special classes of effective Weil divisors.
\begin{definition}
\label{deg.div.def}
\cite{Lai}*{Definition 2.9}
\cite{MR1993750}*{Definition~3.2}
Let $Y$ be a normal variety and let $f \colon Y \to X$ be a proper surjective morphism of normal varieties.
An effective Weil $\mathbb{R}$-divisor $D$ on $Y$ is said to be
\begin{enumerate}
    \item[(i)] 
$f$-exceptional if 
$\codim ({\rm Supp}(f(D))) \geq 2$;

    \item[(ii)] 
of insufficient fibre type for 
$f$ 
if 
$\codim ({\rm Supp}(f(D))) = 1$, 
and there exists a prime divisor 
$\Gamma$ 
on 
$Y$ 
such that 
$\Gamma \not \subseteq {\rm Supp}(D)$ 
and 
${\rm Supp}(f(\Gamma)) \subset {\rm Supp}(f(D))$ 
has codimension one in 
$X$;

	\item[(iii)]
very exceptional for 
$f$
if 
$D = D_1+D_2$ 
with 
$D_1$ 
$f$-exceptional and 
$D_2$
of insufficient fibre type.
\end{enumerate}
\end{definition}

\begin{remark}
\label{q-fact.degen.rmk}
Let $f \colon Y \to X$ be a proper surjective morphism of normal varieties and let $\phi \colon Y' \dashrightarrow Y$ be a birational contraction over $X$, that is, there exists a proper surjective morphism $f' \colon Y' \to X$ and a commutative diagram
\begin{align*}
    \xymatrix{
    Y' \ar@{-->}[rr]^\phi \ar[dr]_{f'} &
    & Y \ar[dl]^{f}\\
    & X &
    }
\end{align*}
Let us denote by 
$E'$ 
the divisorial part of the exceptional locus of 
$\phi$.
If 
$D$ 
is very exceptional for 
$f$ 
on 
$Y$, then 
$D'$ 
is very exceptional for 
$f'$ 
on 
$Y'$, 
where 
$D'$ 
is the strict transform of 
$D$ 
on 
$Y'$. 
Moreover, if 
$f$ 
is of relative dimension 
$1$ 
and 
$D$ 
is 
very exceptional
for 
$f$ 
on 
$Y$, 
then 
$D'+E'$ 
is 
very exceptional
for 
$f'$ 
on 
$Y'$.
\end{remark}

\begin{lemma}
\label{degen.bound.lemma}
Let 
$\mathfrak F$ 
be a log bounded set of morphisms.
Then the set 
$\mathfrak{F}_{{\rm deg}}$ 
of morphism given by 4-tuples of the form 
$(f, Y, B+D, X)$  
such that
\begin{itemize}
    \item 
there exists 
$(f, Y, B, X) \in \mathfrak{F}$, 

    \item 
$D$ 
is reduced and 
$\mathbb{Q}$-Cartier, 
and
	
	\item 
$D$ 
is 
very exceptional for $f$,
\end{itemize} 
is a log bounded set of morphisms.
\\
Moreover, if 
$\mathfrak F$ 
is strongly log bounded, then also 
$\mathfrak{F}_{{\rm deg}}$ 
is strongly log bounded as well.
\end{lemma}

\begin{proof}
By the log boundedness of $\mathfrak F$, there exists a commutative diagram of projective morphisms 
\begin{align}
\label{fam.morf.eqn}
\xymatrix{
(\mathcal Y, \mathcal B) \ar[rr]^{\phi} \ar[dr] & 
&
\mathcal X \ar[dl]\\
 & S &
}
\end{align}
such that for any 4-tuple $(f, Y, B, X)\in \mathfrak{F}$, there is a closed point $s \in S$ and isomorphisms $g \colon \mathcal{X}_s \to X$, $h \colon \mathcal{Y}_s \to Y$ such that the following diagram commutes
\begin{align*}
\xymatrix{
\mathcal{Y}_s \ar[r]^{\phi\vert_{\mathcal{Y}_s}} \ar[d]^{h} & 
\mathcal{X}_s \ar[d]^{g}\\
Y  \ar[r]^{f} &
X,
}
\end{align*}
and $h_\ast \mathcal{B}_s = \supp(B)$. 
If $\mathfrak{F}$ is strongly log bounded, then the only difference is that we can take $\mathcal{B}$ such that $h_\ast \mathcal{B}_s=B$.
\newline
Decomposing $S$ into a finite union of locally closed subsets and passing to a finite cover of $S$, we may assume that for any $s\in S$ there is a 1-1 correspondence between the irreducible components of $\mathcal B$ and those of $\mathcal B_s$.
Decomposing $S$ into a finite union of locally closed subsets, we may assume that $S$ is smooth.
\newline
On 
$\mathcal X$ 
we define the following two Zariski closed subsets
\begin{align*}
\mathcal{J}_\mcx\ & :=
\{x \in \mathcal X \; \vert \; \dim \mathcal{Y}_x > \dim \mathcal Y -\dim \mathcal X
\},\\
\mathcal{I}_\mcx & :=\overline{\{x \in \mathcal X \; \vert \; \mathcal{Y}_x \textrm{ is reducible}\}}^{\rm Zar},    
\end{align*}
where 
$\overline{A}^{\rm Zar}$ 
denotes the Zariski closure of a set 
$A$.
Furthermore, we define
\begin{align*}
\mathcal{J}_\mcy := 
\phi^{-1} (\mathcal{J}_\mcx), \
\mathcal{I}_\mcy := 
\phi^{-1} (\mathcal{I}_\mcx), \
\mathcal{L}_\mcy := 
\mathcal{J}_\mcy \cup \mathcal{I}_\mcy,
\end{align*} 
and we consider the variety $\mathcal{J}_\mcy$ (resp. $\mathcal{J}_\mcx$, $\mathcal{I}_\mcy$, $\mathcal{I}_\mcx$, $\mathcal{L}_\mcy$) as a scheme with the reduced scheme structure.
We will denote by
$\mathcal{J}_{\mcy, s}$ (resp.
$\mathcal{I}_{\mcy, s}$, $\mathcal{L}_{\mcy, s}$)
the schematic fibre of $\mathcal{J}_{\mcy}$ (resp. $\mathcal{I}_{\mcy}$, $\mathcal{L}_{\mcy}$) at $s \in S$.

It follows from Definition~\ref{deg.div.def} and the definitions of the previous paragraph that for any 
$s \in S$ 
if an effective reduced divisor 
$D$ 
on 
$\mathcal{Y}_s$ 
is very exceptional for 
$\phi\vert_{\mys}$ 
then 
$\supp(D) \subset \supp(\mathcal{L}_{\mcy, s})$:
in fact, 
given an irreducible component 
$D'$ 
of 
$D$,
either
$D'$
is  
$\phi\vert_{\mys}$-exceptional
and thus
$D' \subset \mathcal J_{\mathcal X, s}$, 
or 
$D'$
is of insufficient fibre type for 
$\phi\vert_{\mys}$
so that there exists a prime divisor 
$\Gamma \subset \mathcal Y_s$,
$\Gamma \neq D'$ 
such that
$\phi\vert_{\mys}(\Gamma) = \phi\vert_{\mys}(D')$
has codimension one in 
$\mathcal X_s$,
which in turn implies that 
$D' \subset \mathcal I_{\mathcal Y,s}$.
 
Decomposing 
$S$ 
into a finite union of locally closed subsets, we can assume that 
$\mathcal{L}_\mathcal{Y}$ 
is flat over 
$S$. 
Hence, if 
$\mathcal H$ 
is a very ample polarization on 
$\mathcal Y$ 
relatively over 
$S$ 
then there exists a positive constant 
$C=C(\mathfrak F)$ 
such that for any 
$s \in S$,
\begin{align}
\label{degree.ineq}
\deg_{\mathcal{H}_s} \mathcal{L}_{\mathcal Y, s} \leq C,
\end{align} 
where 
$\deg_{\mathcal{H}_s}$ 
indicates the degree with respect to the polarization 
$\mathcal{H}_s$.
As 
$D \subset \mathcal{L}_{\mathcal Y, s}$ 
is a reduced Weil divisor on 
$\mathcal{Y}_s$, 
then~\eqref{degree.ineq} implies that
\begin{align*}
D \cdot \mathcal{H}_s^{\dim \mathcal{Y}_s-1}
= \deg_{\mathcal{H}_s} D \leq deg_{\mathcal{H}_s} \mathcal{L}_{\mathcal Y, s} 
\leq C.
\end{align*}
In view of the existence of Chow varieties, see~\cite{MR1440180}*{\S 1.3}, then it follows that there exists a commutative diagram of projective morphisms 
\begin{align}
\label{fam.morf.deg.eqn}
\xymatrix{
(\mathcal Y', \mathcal B' + \mathcal D') \ar[rr]^{\phi'} \ar[dr] & 
&
\mathcal X' \ar[dl]\\
 & S' &
}
\end{align}
such that for any 4-tuple 
$(f, Y, B+D, X)\in \mathfrak{F}_{\rm deg}$, 
there exists a closed point 
$s' \in S'$ 
and isomorphisms 
$g' \colon \mathcal{X}'_{s'} \to X$, 
$h' \colon \mathcal{Y}'_{s'} \to Y$ 
such that the following diagram commutes
\begin{align*}
\xymatrix{
\mathcal{Y}'_{s'} 
\ar[r]^{\phi'\vert_{\mathcal{Y}'_{s'}}} 
\ar[d]^{h'} & 
\mathcal{X}'_{s'} 
\ar[d]^{g'}\\
Y  \ar[r]^{f} &
X,
}
\end{align*}
and 
$h'_\ast\mathcal{B}'_{s'} = \supp(B)$, 
$h'_\ast\mathcal{D}'_{s'} = D$.
If $\mathfrak{F}$ is strongly log bounded, then the only difference is that we can take $\mathcal{B}'$ such that 
$h'_\ast \mathcal{B}'_{s'}=B$.
In particular, 
$h'_\ast(\mathcal{B}'_{s'} + \mathcal{D}'_{s'}) = B+D$, 
which implies that 
$\mathfrak{F}_{\rm deg}$ 
is strongly log bounded.
\end{proof}

Let us recall the following result that will be useful in dealing with very exceptional divisors.

\begin{lemma}
\label{deg.div.cy.lem}
Let 
$(Y, B)$ 
be a klt pair.
Let 
$f \colon Y \to X$ 
a projective contraction of normal 
$\mathbb{Q}$-factorial 
varieties.
Let 
$D$ 
be an effective Weil 
$\mathbb{R}$-divisor 
on 
$Y$ 
that is very exceptional for
$f$.
Assume that 
$K_Y + B \sim_{\mathbb R, Y} 0$.
Then, there exists a rational contraction $Y \dashrightarrow \hat Y$ over $X$ which contracts exactly the support of $D$.
\end{lemma}

\begin{proof}
Fix 
$0<a\ll 1$ 
such that the pair 
$(Y, B +aD)$ 
is klt. 
We can apply~\cite{MR2929730}*{Theorem~1.8} to obtain the desired result. 
\end{proof}

In Section~\ref{ell.cy.sect}, we will study boundedness of families of elliptic Calabi--Yau varieties with bounded bases.
In order to prove their boundedness up to flops, we will have to show that we can contract very exceptional divisors in log bounded families of morphisms.
Using the above lemma, we can ensure that this result holds.

\begin{proposition}
\label{MMP.deg.bound.prop}
Let 
$d$ 
be a positive integer, 
$\epsilon$ 
be a positive real number, and let 
$\mathcal R \subset \mathbb{Q} \cap (0,1)$ 
be a finite set.
\newline
Let 
$\mathfrak D$
be a log bounded set of log pairs 
$(Y, B)$ 
such that
\begin{enumerate}
    \item 
$(Y, B)$ 
is 
$\mathbb{Q}$-factorial 
projective, 
$\epsilon$-klt 
of dimension 
$d$,
    
    \item
the coefficients of 
$B$ 
are in 
$\mathcal{R}$,

    \item 
$K_Y+B$ 
is semiample.
\end{enumerate}
\noindent
Then, there exists a strongly log bounded set of morphisms 
$\mathfrak{F}_{{\rm deg, min}}$ 
whose elements are 4-tuples 
$(f', Y', B', X)$ 
which satisfy the following property:
\newline
for any 
$(Y, B) \in \mathfrak D$ 
and any divisor 
$D$ 
that is very exceptional for the Iitaka fibration 
$f \colon Y \to X$ 
of 
$K_Y+B$,
there exists a 4-tuple
$(f', Y', B', X) \in \mathfrak{F}_{{\rm deg, min}}$ 
such that
\begin{enumerate}
    \item[(i)] 
$X:= {\rm Proj}(\oplus_{i=0}^\infty H^0(Y, \mathcal{O}_Y(i(K_Y+B))))$, and

    \item[(ii)]
$f' \colon Y' \to X$ is a good minimal model for $D$ over $X$ and $B'$ is the strict transform of $B$ on $Y'$.
\end{enumerate}
\end{proposition}

Under the assumptions of the proposition, as $K_Y+B \sim_{\mathbb{Q}, f} 0$, then $K_Y+B+ \lambda D \sim_{\mathbb{Q}, f} \lambda D$, for all $\lambda \in \mathbb{R}$, and the existence of a good minimal model $Y'$ for $D$ over $X$ is implied by Lemma~\ref{deg.div.cy.lem}.

\begin{proof}
By Lemma~\ref{bound.morph.lemma}, the set 
$\mathfrak F$ 
of 4-tuples 
$(f, Y, B, X)$ 
such that  
\begin{itemize}
    \item 
$(Y, B) \in \mathfrak D$, 
and

    \item 
$f \colon Y \to X$ 
is the Iitaka fibration of 
$K_Y+B$, 
\end{itemize}
is a strongly log bounded set of morphisms.
By Lemma~\ref{degen.bound.lemma}, any divisor 
$D$ 
that is very exceptional for 
$f$ 
is bounded on 
$Y$. 
Hence, the set 
$\mathfrak{F}_{\rm{deg}}$ 
of 4-tuples 
$(f, Y, B+D, X)$ 
such that
\begin{itemize}
    \item 
$(Y, B) \in \mathfrak D$,

    \item 
$f \colon Y \to X$ 
is the Iitaka fibration of 
$K_Y+B$, 
and

    \item 
$D$ 
is very exceptional for 
$f$,
\end{itemize}
is a strongly log bounded set of morphisms and there is a commutative diagram
\begin{align}
\label{diag.bound.fam.proof.eqn}
    \xymatrix{
    (\mathcal{Y}, \mathcal B+ \mathcal D) 
    \ar[rr]^-{\phi} \ar[rd] &
    & \mathcal X\ar[ld]\\
    & S &
    }
\end{align}
such that 
\begin{itemize}
\item 
$\mathcal X$ 
is the Iitaka fibration of 
$K_\mathcal{Y}+\mathcal{B}$ 
over 
$S$,

\item 
for any 4-tuple 
$(f, Y, B+D, X) \in \mathfrak{F}_{\rm{deg}}$ 
there exists a closed point 
$s \in S$ 
and isomorphisms 
$g \colon \mathcal{X}_{s} \to X$, 
$h \colon \mathcal{Y}_s \to Y$ 
such that the following diagram commutes
\begin{align}
\label{comm.square.diag.eqn.1}
\xymatrix{
\mathcal{Y}_s \ar[r]^{\phi\vert_{\mathcal{Y}_s}} \ar[d]^{h} & 
\mathcal{X}_s \ar[d]^{g}\\
Y  \ar[r]^{f} &
X,
}
\end{align}
and 
$h_\ast (\mathcal{B}_s+ \mathcal{D}_s) = B+D$.
\end{itemize}
Decomposing $S$ into a finite union of locally closed subsets and possibly discarding some components, we may assume that every fibre $\mathcal{Y}_s$ is a variety and that $\supp(\mathcal B+\mathcal{D})$ does not contain any fibre. 
Decomposing $S$ into a finite union of locally closed subsets and passing to a finite cover of $S$, we may assume that for any $s\in S$ and there is a 1-1 correspondence between the irreducible components of $\mathcal B+\mathcal{D}$ and those of $\mathcal{B}_s+\mathcal{D}_s$.
Decomposing $S$ into a finite union of locally closed subsets, we may assume that $S$ is smooth;
up to shrinking $S$, we can also assume that the set 
\begin{align*}
   S':= 
   \{s\in S \ \vert \ \exists (f, Y, B+D, X) \in \mathfrak{F}_{{\rm deg}} 
   \textrm{ such that } 
   (Y, B+D) \simeq (\mathcal{Y}_s, \mathcal{B}_s+\mathcal{D}_s)\}
\end{align*}
is Zariski dense in $S$.
Applying~\cite{1810.01990}*{Proposition 2.10} to the pairs 
$(Y, B)$ 
in each of the 4-tuples 
$(f, Y, B+D, X) \in \mathfrak{F}'$, 
we may assume that 
$\mathcal Y$ 
is 
$\mathbb{Q}$-factorial, 
possibly after further decomposing 
$S$ 
again into a finite union of locally closed subsets.

\medskip

{\bf Claim}. 
{\it Up to decomposing 
$S$ 
into a finite union of locally closed subsets and discarding those components that do not contain points of 
$S'$, 
we may assume that there exists a positive rational number 
$\delta$ 
such that 
$(Y, B+\delta D)$ 
is 
$\frac \epsilon 2$-klt 
for any 4-tuple 
$(f, Y, B+D, X) \in \mathfrak{F}'$}.

\begin{proof}[Proof of the Claim]
Given a point 
$s \in S'$, 
there exists a positive rational number 
$\delta'$ 
such that 
$(\mathcal{Y}_s, \mathcal{B}_s + \delta' \mathcal{D}_s)$ 
is 
$ \frac \epsilon 2$-klt.
As 
$\mathcal Y$ 
is 
$\mathbb Q$-factorial, 
then 
$(\mathcal Y, \mathcal B + \delta' \mathcal D)$ 
is 
$ \frac \epsilon 2$-klt 
in a neighborhood of 
$\mathcal{Y}_s$ 
in 
$\mathcal Y$, 
as being 
$\frac \epsilon 2$-klt 
is an open property.
Thus, decomposing 
$S$ 
into a finite union of locally closed subsets we can find a uniform choice of 
$\delta$ 
such that 
$(\mathcal Y, \mathcal B + \delta \mathcal D)$ 
is 
$ \frac \epsilon 2$-klt, which proves the desired claim. 
By noetherian induction and up to discarding those components of $S$ not containing points of $S'$, we can still assume that $S'$ is dense in $S$ and repeat the same argument, which then concludes the proof of the claim.
\end{proof}
\noindent
In view of the claim, we can run the 
$(K_\mathcal Y+ \mathcal B + \delta \mathcal D)$-MMP
over 
$\mathcal{X}$ 
\begin{align}
\label{diag.bound.fam.proof.eqn3}
\xymatrix{
\mathcal Y =:
\mathcal Y_0 \ar@{-->}[r]^{\pi_0} \ar[drrr]_{\phi=:\phi_0}&
\mathcal Y_1 \ar@{-->}[r]^{\pi_1} \ar[drr]^{\phi_1}&
\mathcal Y_2 \ar@{-->}[r]^{\pi_2} \ar[dr]^{\phi_2}&
\dots \ar@{-->}[r]^{\pi_{i-1}} &
\mathcal Y_i \ar@{-->}[r]^{\pi_i} \ar[dl]_{\phi_i}&
\mathcal Y_{i+1} \ar@{-->}[r]^{\pi_{i+1}} \ar[dll]^{\phi_{i+1}}&
\dots
\\
& & &
\mathcal X &
& & &
}
\end{align}
which must terminate with a good minimal model 
\begin{align}
\label{diag.bound.fam.proof.eqn2}
    \xymatrix{
(\mathcal{Y}, \mathcal B+ \mathcal D) 
\ar@{-->}[rrrr]^-{\pi:= \pi_{l-1} \circ \pi_{l-2}\circ \dots \pi_2 \circ \pi_1 \circ\pi_0} 
\ar[rrd]_\phi 
    & & & &
(\mathcal{Y}_l, \pi_\ast (\mathcal{B}+ \mathcal{D}) )
=:(\mathcal{Y}', \mathcal{B}'+ \mathcal{D}') 
\ar[lld]^{\phi_l=:\phi'}
    \\
& &
\mathcal X\ar[d] 
& & 
 	\\
& &
S 
& &
}
\end{align}
by~\cite{HX}*{Theorem 1.1}, since $\mathcal D$ is vertical over $\mathcal X$ by construction: indeed, $\mathcal D\vert_{\mathcal{Y}_s}$ is vertical over $\mathcal{X}_s$ for any $s \in S'$. 
By decomposing 
$S$ 
into a finite union of locally closed subsets and using Noetherian induction, we can assume that at each step 
$\xymatrix{
\mathcal Y_i \ar[r]^{\pi_i} & 
\mathcal Y_{i+1}
}$
of the MMP in~\eqref{diag.bound.fam.proof.eqn3}, the exceptional locus of 
$\pi_i$ 
is either empty or horizontal over every component of 
$S$, 
and that if 
${\rm exc}(\pi_i)$ 
is divisorial (resp. has codimension at least $2$), then the same holds for the restriction of 
${\rm exc}(\pi_i\vert_{\mathcal Y_{i, s}})$ 
along each fibre over 
$s \in S$ 
that intersects the exceptional locus of 
$\pi_i$. 
From these assumptions and observations, we can conclude that for all 
$s \in S'$, 
$(\mathcal{Y}'_s, \mathcal{B}'_s+\mathcal{D}'_s) \to \mathcal{X}_s$ 
is a good relatively minimal model of 
$(\mathcal{Y}_s, \mathcal{B}_s + \mathcal{D}_s)$ 
over 
$\mathcal{X}_s$.
But, as 
$s \in S'$, 
$\mathcal D_s$ is very exceptional for $f$, by the definition of $S'$.
Thus,
$\mathcal D_s$ 
must be contracted on every good minimal model of 
$(\mathcal Y_s, \mathcal B_s+\mathcal D_s)$
by Lemma~\ref{deg.div.cy.lem}.
\end{proof}

\subsection{Canonical bundle formula}
\label{sect.cbf}

As our main focus in this paper will be on Mori fibrations of log Calabi--Yau pairs, we collect in this subsection some results about lc-trivial fibrations which can be applied to our setting. 
For more details, we refer the reader to~\cites{MR2153078, MR2134273}, and to~\cite{MR2944479}*{Theorem 3.1} for the case of real coefficients.   
For the definition of b-divisor and related notions, we refer the reader to~\cite{bir-zh}.

Given a log pair $(X,B)$ and a lc-trivial (resp. klt-trivial) fibration, that is, a contraction $f \colon X \to Z$ of normal quasi-projective varieties such that $(X, B)$ is log canonical (resp. klt) at the generic point of $Z$, and $K_X + B \sim_\mathbb{R} f^* L$, then we can define a divisor $B_Z$ on $Z$ by posing 
\begin{align}
\label{def.bound.part.eqn}
B_Z := \sum (1-l_D) D,
\end{align} 
where the sum is taken over every prime divisor $D$ on $Z$, and $l_D$ is the log canonical threshold of $f^*D$ with respect to $(X,B)$ over the generic point of $D$.
We can also define a divisor $M_Z$ on $Z$, by posing
\begin{align}
\label{def.moduli.part.eqn}
M_Z := L-(K_Z+B_Z).
\end{align}
By these definitions, we get the following formula, usually referred to as the canonical bundle formula, 
\begin{align}
\label{cbf.def.eqn}
K_{X}+B \sim _{\mathbb{R}} f^*(K_{Z}+B_{Z}+M_{Z}).
\end{align}

\begin{remark}
\label{cbf.lin.eq.rmk}
It is actually not hard to see that the following more refined version of~\eqref{cbf.def.eqn} holds when $B_Z$ is a $\mathbb{Q}$-divisor,
\begin{align*}
r(K_X+B) \sim (f^\ast r(K_Z+B_Z+M_Z)),
\end{align*}
where 
\begin{align*}
r := \min \{m \in \mathbb{N}_{>0} \; | \; m(K_F + B|_F) \sim 0, \; \text{for a general fibre } F\},
\end{align*}
cf.~\cite{MR3180598}*{Remark 2.10}. 
\end{remark}

Given a higher model  
$Z' \to Z$ 
of 
$Z$, 
considering a sufficiently high resolution of indeterminacies
$q \colon X' \to X$
of the induced rational map
$X \dashrightarrow Z'$
and the sub-pair 
$(X', B')$
-- i.e., 
$B'$ 
may have negative coefficients --
 where 
$B'$ 
is the sub-boundary obtained by log-pullback of 
$K_X+B$ 
on 
$X'$, 
$K_{X'}+B' = q^\ast (K_X+B)$,
it is possible to define divisors 
$B_{Z'}$, 
$M_{Z'}$ 
in the same way as it was done above, see~\cite{MR2134273}*{\S~3}.
The collections of these divisors, when 
$Z'$ 
varies among higher models of 
$Z$, 
define 
b-$\mathbb{R}$-divisors 
$\mathbf{B}$, 
$\mathbf{M}$ 
whose traces on 
$Z$ 
are 
$B_Z$, $M_Z$, 
respectively.
The b-divisors 
$\mathbf{B}$, 
$\mathbf{M}$ 
are referred to as the boundary and the moduli part of 
$f$,
respectively.
In the rest of the paper, we will refer to $B_Z$ as the boundary divisor and to $M_Z$ as the moduli part of the fibration $f$ on $Z$.
 
When the singularities of the pair $(X, B)$ are mild, then one can describe more precisely the singularities of $(Z, B_Z+M_Z)$.  
\begin{theorem}
~\cite{MR2944479}*{Theorem~3.1}
\label{bundle}
Let $f \colon X\rightarrow Z$ be a contraction of normal varieties and
let $(X, B)$ be a log canonical pair that is klt over the generic point of $Z$ and $K_{X}+B \sim_{f,\mathbb{R}} 0$.
\newline
If $(X, B)$ is klt, then there exists a choice of an effective $\mathbb{R}$-divisor $M \sim_\bR M_Z$ such that $(Z,B_{Z}+M)$ is klt.
\end{theorem}

\begin{remark}
\label{cbf.rmk}
We will use the same notation as in the theorem above.
When 
$B$ 
is a divisor with rational coefficients and 
$(X, B)$ 
is klt, the above result was first proved by Ambro,~\cite{MR2134273}, who also proved that there exists a sufficiently high birational model 
$Z'$ 
of 
$Z$ 
on which 
$\mathbf{M}$ 
descends to 
$M_{Z'}$ 
which is a nef 
$\mathbb{Q}$-divisor.
This is usually summarized by saying that the moduli part 
$\mathbf M$ 
is a b-nef 
b-$\mathbb Q$-divisor.
Later, in~\cite{MR3329677}, the same result was extended to the case of lc pairs.
\newline
Under such assumption, it is an immediate consequence that 
$(Z' \to Z, B_Z, M_{Z'})$ 
is a generalised pair in the sense of \S~\ref{gen.pair.sect}.
\newline
If 
$B$ 
is an 
$\mathbb{R}$-divisor, 
it is not true in general that 
$\mathbf{M}$ 
is going to be a b-nef 
b-$\mathbb R$-divisor, 
but rather its trace on any model of 
$X$ 
will be pseudoeffective.
\newline
Nonetheless, by writing 
$K_X+B$ 
as a convex sum of rational klt log divisors, i.e., 
$K_X+B=\sum_{i=1}^s r_i(K_{X}+B_i)$, 
where for all 
$i=1, \dots, s$,
$r_i > 0$, 
$B_i$
is a 
$\mathbb{Q}$-divisor,
$K_X+B_i\sim_\mathbb{Q}0$, 
and 
$\sum_{i=1}^s r_i=1$,
we can then write 
$K_X+B \sim_\mathbb{R} f^\ast (K_Z+C_Z+N)$, 
where 
$N$ 
is the trace on 
$Z$ 
of a b-nef 
b-$\mathbb R$-divisor 
$\mathbf N$ 
obtained as the convex sum of the moduli parts of the log divisors 
$K_X+B_i$, 
while 
$C_Z$ 
is the convex sum of the boundary parts of the log divisors 
$K_X+B_i$ 
effective divisor.
On the other hand, this kind of operation is non-canonical and thus there is no way to enforce uniqueness of 
$C_Z$ 
and 
$\mathbf N$, 
cf.~\cite{MR4398255}.
\end{remark}

It is conjectured that when 
$(X, B)$ 
is klt and $B$ 
has rational coefficients, then on a sufficiently high model 
$Z' \to Z$ 
the trace 
$M_{Z'}$ 
of the moduli part 
$\mathbf M$ 
on 
$Z'$
is semi-ample.
Actually, the following much stronger statement, known as the effective semi-ampleness conjecture of the moduli part, is expected to hold;
it first appeared in~\cite{MR2448282}*{Conjecture~7.13}.

\begin{conjecture}
\label{conj:eff.semi}
Let 
$k, r$ 
be positive integers.
There exists a positive integer 
$m=m(k,r)$ 
such that for any klt projective pair 
$(X, B)$, 
where 
$B$ 
has rational coefficients, and any klt-trivial fibration 
$f \colon X \rightarrow Z$ 
as in Theorem~\ref{bundle}, where the relative dimension of 
$f$ 
is 
$k$ 
and 
$r$ 
is the positive integer defined in Remark~\ref{cbf.lin.eq.rmk}, there exists a birational model 
$Z' \rightarrow Z$ 
where the multiple 
$mM_{Z'}$ 
of the moduli b-divisor on 
$Z'$ 
descends and is base point free.
\end{conjecture}

The effective semi-ampleness conjecture remains a hard unsolved problem, but for our purposes we only need the case where $M_Z$ is numerically trivial and the generic fibre of $f$ is smooth: this was settled in~\cite{MR3180598}.

\begin{theorem}
\label{flor.thm}
\cite{MR3180598}*{Theorem 1.3}
Fix a positive integer $b$.
There exists an integer $m=m(b)$ such that for any klt-trivial fibration $f\colon (X, B)\rightarrow Z$ with
\begin{itemize}
\item 
$B$ a Weil $\mathbb{Q}$-divisor,

\item 
$M_Z\equiv 0$, and

\item 
$\dim H^{\dim E}(E,\mathbb{C})=b$ for a non-singular model $E$ of the cover $E'\rightarrow F$ associated to the unique element of $|r(K_X+B)|_F |$ of a general fibre $F$ of $f$, where $r$ is the positive integer defined in Remark \ref{cbf.lin.eq.rmk},
\end{itemize}

we have that $mM_Z\sim 0$.
\end{theorem}

The statement of Theorem~\ref{flor.thm} is slightly different than that of the effective semi-ampleness conjecture: 
in fact, while in Theorem~\ref{flor.thm} the integer 
$m$ 
depends on the dimension 
$b$ 
of the middle cohomology of the resolution of a finite cover of a general fibre, in Conjecture~\ref{conj:eff.semi} the integer 
$m$ 
depends on the relative dimension 
$k$ 
of 
$f$ 
and on the integer 
$r$.
While it is clear from the statement of Theorem~\ref{flor.thm} that the integer 
$b$ 
is inherently related to the integer 
$r$, 
as 
$E$ 
is a resolution of the ramified cover of degree 
$r$ 
induced by 
$r(K_F+B\vert_F) \sim 0$, 
it is not evident that this result provides a full solution to Conjecture~\ref{conj:eff.semi} when the moduli part is numerically trivial.
Nonetheless, when the log pair induced on the general fibre of the klt-trivial fibration 
$f$ 
belongs to a bounded family, using Theorem~\ref{flor.thm} and standard techniques in the theory of bounded pairs, we can show that the positive integer 
$m$ 
that trivializes the moduli part 
$M_Z$ 
can be chosen to be uniformly bounded.

\begin{corollary}
\label{cbf.triv.cor}
Fix a DCC set 
$\mathcal{R} \subset (0, 1) \cap \bQ$. 
Let 
$\mathfrak{D}$ 
be a bounded set of klt pairs.
Assume that for any pair 
$(X, B) \in \mathfrak{D}$, 
the coefficients of 
$B$ 
are in 
$\mathcal{R}$,
and 
$K_X+B \sim_\bQ 0$.
Then there exists an integer number 
$m=m(\mathfrak{D}, \mathcal{R})$ 
such that for any klt-trivial fibration 
$f\colon (X,B)\rightarrow Z$ 
with
\begin{itemize}
\item 
$M_Z\equiv 0$, 
and

\item 
the pair 
$(F, B|_F) \in \mathfrak{D}$, 
where 
$F$ 
is a general fibre of 
$f$,
\end{itemize}
we can choose a divisor 
$M$ 
in the class 
$M_Z$ 
with 
$mM\sim 0$.
\end{corollary}

\begin{proof}
As 
$\mathcal{R}$ 
is DCC, by~\cite{MR3224718}*{Theorem 1.5} there exists a finite subset 
$\mathcal{R}_0 \subset \mathcal{R}$ 
such that all coefficients of pairs in 
$\mathfrak{D}$ 
belong to 
$\mathcal{R}_0$.
Moreover,~\cite{DCS}*{Corollary 2.9} implies that there exists 
$\epsilon >0$ 
such that all pairs in 
$\mathfrak{D}$ 
are 
$\epsilon$-klt.
Hence, by Proposition~\ref{fam.good.prop} we can assume that 
$\mathfrak D$ 
is strongly log bounded. 
In particular, there exists a morphism of quasi-projective varieties 
$\mathcal{Z} \to S$ 
and an effective 
$\mathbb{Q}$-divisor 
$\mathcal{B}$ 
on 
$\mathcal Z$ 
such that for any pair 
$(X, B) \in \mathfrak{D}$ 
there exists 
$s \in S$ 
and an isomorphism 
$f_s \colon \mathcal{Z}_s\to X$ 
and 
$\mathcal{B}_s = f^\ast_s B$.
Moreover, up to decomposing 
$S$ 
into a disjoint union of finitely many locally closed subsets, we can assume that on each component of 
$S$ 
there exists a positive integer 
$l$ 
such that 
$l(K_{\mathcal{Z}} + \mathcal{B}) \sim_{S} 0$. 
In fact, constructing a log resolution 
\[
\xymatrix{
\mathcal{Z}' \ar[rr] \ar[dr]& & \mathcal{Z}\ar[dl]\\
& S &
}
\]
of 
$(\mathcal{Z}, \mathcal{B})$ 
which is log smooth over the base $S$ as in the proof of~\ref{fam.good.prop}, and using~\cite{MR3779687}*{Theorem 4.2} there must exist a positive integer 
$n$ 
such that 
$n(K_X+B) \sim 0$ 
for any pair in 
$\mathfrak{D}$.
Up to stratifying again
$S$ 
in the Zariski topology, we can construct a bounded set of varieties 
$\mathfrak{C}$
whose elements are smooth projective varieties 
$E$ 
that are non-singular models of the cover 
$E' \to X$ 
associated to the unique element of 
$|n(K_X+B)|$. 
The statement of the corollary then follows by noticing that since 
$\mathfrak{C}$
is bounded then there exists a natural number 
$b=b(\mathfrak{C})$ 
such that for any 
$E \in \mathfrak{C}$, $h^{\dim E}(E) \leq b$, by Ehresmann's Theorem~\cite{MR2451566}*{Theorem~9.3}. 
The conclusion then follows from Theorem~\ref{flor.thm} via noetherian induction. 
\end{proof}

Similarly to the classical case, we can define a canonical bundle formula for generalised pairs. Let $(X,B, \mathbf M)$ be a projective generalised pair, and let $f \colon X \to Z$ be a contraction where $\dim Z > 0$. 
We shall assume that $(X,B, \mathbf M)$ is generalised log canonical over the generic point of $Z$, $B$ is a $\mathbb Q$-divisor, $\mathbf M$ is a $\mathbb Q$-b-divisor and that $K_X + B + M \sim_{f, \mathbb{Q}} 0$. 
We shall also fix a $\mathbb Q$-divisor $L$ on $Z$ such that $K_X + B + M \sim_\mathbb{Q} f^\ast L$. 
Given any prime divisor $D$ on $Z$, let $l_D$ be the generalised log canonical threshold of $f^\ast D$ with respect to $(X,B+M)$ over the generic point of $D$. 
Then, we define $G_Z:= \sum c_D D$, where $c_D : = 1-l_D$ and $N_Z : = L - (K_Z + B_Z)$, so that 
\begin{align}\notag
K_X + B + M \sim_\mathbb{Q} f^\ast (K_Z + G_Z + N_Z).
\end{align}

The divisor $\mathbb Q$-divisor $G_Z$ (resp. $N_Z$) is just the trace on $Z$ of a $\mathbb Q$-b-divisor $\mathbf{G}$ (resp.$\mathbf{N}$) and it was proven in~\cite{1807.04847} that the two $\mathbb Q$-b-divisors $\mathbf{G}, \mathbf N$ induce a structure of generalised pair on the base of a relatively trivial fibration.
Finally, this result was generalized in~\cite{MR4398255} to the case of $\mathbb R$-divisors, with some extra assumption on the type of moduli part that is allowed.

\subsection{The different of a section of a fibration.}
\label{subs.diff.sect}

In this subsection, we introduce a result that will be used in the proof of Theorem~\ref{ell.cy.bound.base.thm}.

Let us first recall the notion of {\it different}.
Let $\Sigma$ be a reduced and irreducible divisor and $B$ be an effective divisor on a normal quasi projective variety $Y$.
We denote by $\nu \colon \Sigma^\nu \to \Sigma$ the normalization of $\Sigma$.
We assume that $\Sigma$ has no common components with $B$ and that $(Y, \Sigma+B)$ is a log pair in codimension 2 on $Y$, i.e., there exists a Zariski closed subset $Z \subset Y$ with ${\rm codim}_Y Z >2$ such that $K_Y+\Sigma+B$ is an $\mathbb R$-Cartier divisor when restricted to $Y \setminus Z$.
Then there exists a canonically defined effective $\mathbb{R}$-divisor $\diff_{\Sigma^\nu}(B)$ on $\Sigma^\nu$, called the \textit{different} of $B$, defined as follows -- see~\cite{MR1225842}*{\S 16} for details.
By taking hyperplane cuts and applying adjunction, it suffices to consider the case where $Y$ is a surface.
There exists a log resolution of singularities for the surface pair $(Y, \Sigma)$, $r \colon \bar{Y} \to Y$, such that the strict transform $\tilde{\Sigma}$ of $\Sigma$ coincides with the normalisation $\Sigma^\nu$ of $\Sigma$.
Hence, there exists an $\bR$-divisor $B_{\bar{Y}}$ on $\bar{Y}$ such that 
\begin{align*}
K_{\bar{Y}}+\Sigma^\nu+ B_{\bar{Y}} = r^\ast(K_Y+\Sigma+B), 
\qquad
\text{ and } 
\qquad
r_\ast B_{\bar{Y}} =B. 
\end{align*}
The different $\diff_{\Sigma^\nu}(B)$ is then defined as the $\bR$-divisor $(B_{\bar{Y}}) |_{\Sigma^\nu}$.
Moreover, by restricting to the complement in $\Sigma^{\nu}$ of a Zariski closed subset $W \subset \Sigma^\nu$ with ${\rm codim}_{\Sigma^\nu} W \geq 2$, 
$K_{\Sigma^\nu}+\diff_{\Sigma^\nu}(B)=\nu^*((K_Y+\Sigma+B) |_\Sigma)$, by construction and adjunction.

We will now consider a contraction of normal quasi-projective varieties 
$f \colon Y \to X$ 
of relative dimension one, together with a rational section 
$s \colon X \dashrightarrow Y$.

\begin{lemma}\label{section.lemma}
Let 
$f \colon Y \to X$ 
be a surjective morphism of normal quasi-projective varieties with 
$\dim Y -\dim X=1$.
Assume that 
$(Y, 0)$ 
is terminal and that there exists a rational section 
$s \colon X \dashrightarrow Y$.
Let 
$\Sigma$ 
be the Zariski closure of 
$s(B)$ 
and let 
$\nu\colon \Sigma^\nu \to \Sigma$ 
be the normalization of 
$\Sigma$. 
Then, 
$\diff_{\Sigma^\nu}(0)$ 
is exceptional over 
$B$.
\end{lemma}

Let us remark here that since 
$(Y, 0)$ 
is assumed to be terminal, 
then 
$Y$ 
is smooth in codimension 
$2$, 
and hence 
$(Y, \Sigma)$ 
is a log pair in codimension 
$2$, 
thus satisfying the assumption made above in the definition of different.

\begin{proof}
See~\cite{DCS}*{Proof of Lemma~5.1}.
\end{proof}

\section{Rationally connected log Calabi--Yau pairs}
\label{rc.lcy.sect}
\subsection{Towers of Fano fibrations and boundedness} \label{towers.sect}
We briefly recall some important results from~\cites{1811.10709, DCS}.

Given a klt log Calabi-Yau pair $(X, B)$ with $B \neq 0$, any run of a $K_X$-MMP 
\begin{align*}
X \dashrightarrow X' \to Z
\end{align*}
terminates with a Mori fibre space structure $f' \colon X' \to Z$, see Theorem~\ref{mmp}.
By the canonical bundle formula, cf. \S~\ref{sect.cbf}, $Z$ carries a structure of log Calabi--Yau pair $(Z, \Gamma)$ as implied by Theorem~\ref{bundle}.
If $K_Z \sim_{\mathbb{Q}} 0$, i.e. $\Gamma=0$, then we say that the pair $(X, B)$ is of product-type, see~\cite{DCS}*{Definition 2.23}. 
Otherwise, $\Gamma > 0$ and, assuming $K_Z$ is $\mathbb{Q}$-factorial, we can run a $K_Z$-MMP in turn and repeat the same analysis as above.
The $\mathbb{Q}$-factoriality of $Z$ is not a very strong assumption, as, for example, it is readily implied by $X$ being $\mathbb{Q}$-factorial.

By iterating this procedure and using the so-called two-ray game, see~\cite{MR1225842}*{Chapter 5}, the following description of Calabi-Yau pairs was given in~\cite{DCS}*{Theorem 3.2}.
\begin{theorem} 
\label{structure}
Let $(X, B)$ be a klt Calabi-Yau pair with $B \neq 0$. 
Then there exists a birational contraction 
\begin{align*}
		\pi \colon X \dashrightarrow X'
\end{align*}
to a $\mathbb{Q}$-factorial Calabi-Yau pair $(X', B' =\pi_\ast B)$, $B'\neq 0$ and a tower of morphisms
\begin{align}
\label{mfs.tower.diag}
\xymatrix{
			X'=X_0 \ar[r]^{p_0} & 
			X_1 \ar[r]^{p_1} & 
			X_2 \ar[r]^{p_2} &
			\dots \ar[r]^{p_{k-1}}& 
			X_{k}
}
\end{align}		
\noindent such that 
\begin{enumerate}
		\item[(i)] 
for any $1 \leq i < k$ there exists a boundary $B_i \neq 0$ on $X_i$ and $(X_i, B_i)$ is a klt Calabi-Yau pair,

 		\item[(ii)] 
for any $0 \leq i < k$ the morphism $p_i\colon X_i \to X_{i+1}$ is a Mori fibre space, with $\rho(X_i/X_{i+1})=1$, and

		\item[(iii)]  
either $\dim X_k=0$, or $\dim X_k >0$ and $X_k$ is a klt variety with $K_{X_k} \sim_\mathbb{Q} 0$.
\end{enumerate}		
\end{theorem}

When $\dim X_k >0, \; K_{X_k} \sim_\mathbb{Q} 0$, then we say that $(X, B)$ is of product type, see~\cite{DCS}*{Definition 2.23}.

Using Theorem~\ref{structure}, the strategy of~\cite{DCS}, and the techniques of~\cites{Bir16a, Bir16b}, Birkar has shown that log Calabi--Yau pairs with bounded singularities admitting a tower of fibration as in~\eqref{mfs.tower.diag} are log bounded, provided we assume that the last element of the tower belongs to a bounded family.

\begin{theorem}
\cite{1811.10709}*{Theorem 1.4}
\label{birk.tower.thm}
Let $d, r$ be natural numbers, $\epsilon, \tau$ be positive real numbers.
Consider pairs $(X, B)$ and contractions $f \colon X \to Z$ such that 
\begin{itemize}
    \item 
$(X, B)$ 
is projective 
$\epsilon$-lc 
of dimension 
$d$ 
and 
$B$ 
is an $\mathbb R$-divisor,

    \item 
$K_X+B \sim_\mathbb{R} f^\ast L$ 
for some 
$\mathbb{R}$-divisor 
$L$,
    
    \item 
the coefficients of 
$B$ 
are at least 
$\tau$,

    \item 
$f$ 
factors as a sequence of non-birational contractions
    \begin{align}
\label{mfs.tower.eq}
X = X_1 \to X_2 \to \dots \to X_k = Z,
    \end{align}
    
    \item 
for each 
$i$, 
$-K_{X_i}$ 
is ample over 
$X_{i+1}$,
    
    \item 
there is a very ample divisor 
$A$ 
on 
$Z$ 
with 
$A^{\dim Z} \leq r$, 
and 

    \item 
$A-L$ 
is ample.
\end{itemize}
Then the set of such $(X, B)$ forms a log bounded family.
\end{theorem}

Combining the two results above, Birkar also showed that boundedness holds for log Calabi--Yau pairs of non-product type.

\begin{theorem}
\cite{1811.10709}*{Theorem 1.5}
\label{lcy.fib.non-prod.thm}
Let $d$ be a natural number and $\epsilon,\tau$ be positive real numbers. 
Consider pairs $(X,B)$ with the following properties: 
\begin{itemize}
\item $(X,B)$ is projective $\epsilon$-lc of dimension $d$,

\item $K_X+B\sim_\mathbb{R} 0$, 

\item $B\neq 0$ and its coefficients are $\ge \tau$, and

\item $(X,B)$ is not of product type.   
\end{itemize}
Then the set of such $(X,B)$ is log bounded up to isomorphism in codimension one.
\end{theorem}

To obtain this result, it is necessary to bound the singularities of the pair together with the coefficients of the boundary $B$.
While the former condition is unavoidable as already clear in the case of singular del Pezzo surfaces, the latter is a technical condition that it should be possible to waive, as predicted by Conjecture~\ref{eps.lcy.conj}.
Indeed, this is what we achieve in Theorem~\ref{delta.lcy.thm} at the expense of fixing the torsion index of $K_X+B$.
A first step in this direction is represented by Theorem~\ref{mfs.tow.CYend.thm} below.

As an immediate corollary to Theorem~\ref{birk.tower.thm}, we get the following log boundedness result that will be useful in the next subsection.

\begin{corollary}\label{logCY.fibred.cor}
Fix $d$ and $\Phi \subset [0,1)$ a DCC set of rational numbers. 
Then the set of klt log CY pairs $(X, B)$ such that 
\begin{itemize}
    \item $(X, B)$ is projective klt of dimension $d$,
    \item $K_X+B \sim_\mathbb{Q} 0$,
    \item the coefficients of $B$ are in $\Phi$,
    \item the constant map $X \to \{pt. \}$ factors as a sequence of non-birational contractions
    \begin{equation}\label{mfs.tower.pt.eq}
        X = X_1 \to X_2 \to \dots \to X_l = \{pt. \},
    \end{equation}
    and,
    \item for each $i, \; -K_{X_i}$ is ample over $X_{i+1}$.
\end{itemize}
Then the set of such $(X, B)$ forms a log bounded family.
\end{corollary}

\begin{proof}
As $\Phi$ is a DCC subset, then~\cite{MR3224718}*{Theorem 1.1} implies that there exists $\epsilon=\epsilon(d, \Phi)$ such that all the pairs $(X, B)$ here considered are $\epsilon$-klt, see~\cite{CDHJS}*{Lemma 3.12}.
The result then follows by applying Theorem~\ref{birk.tower.thm} with $Z$ equal to a point and $l=d$, cf. also~\cite{1811.10709}*{\S 7}.
\end{proof}

\subsection{Boundedness with fixed torsion index}

As we have seen above, there is no boundedness result currently for the case of log Calabi--Yau pairs of product type.
On the one hand, this is not surprising, as, for example, among such pairs there are those of the form $(S \times \mathbb{P}^1, p_2^\ast(0 + \infty))$, where $S$ is a K3 surface and $p_2$ is the projection to the second factor: these log Calabi--Yau pairs cannot possibly be bounded, as K3 surfaces are not bounded.
On the other hand, Conjecture~\ref{eps.lcy.conj} predicts that if the total space of the pair is rationally connected, then even in the product-type case we should expect boundedness.
As already mentioned, our aim is to take care of those log Calabi--pairs that are of product-type, as those have yet to be fully understood as far as their boundedness goes.
To this end, we prove that when the torsion index of $K_X+B$ is bounded on the total space of a product-type log Calabi--Yau pair $(X, B)$ and $X$ is endowed with a tower of Mori fibre spaces terminating with a K-trivial variety $Z$, then also the torsion index of $K_Z$ is bounded.

\begin{theorem}
\label{mfs.tow.CYend.thm}
Fix $d, l$ positive integers.
Consider pairs $(X, B)$ and contractions $f \colon X \to Z$ such that
\begin{itemize}
    \item $(X, B), \; B>0$ is a klt projective pair of dimension $d$,
    \item $l(K_X+B) \sim 0$,
    \item $f$ factors as a sequence of non-birational contraction
    \begin{equation}
    \label{mfs.tower.eq2}
        X = X_1 \to X_2 \to \dots \to X_k = Z,
    \end{equation}
     \item for each $i, \; -K_{X_i}$ is ample over $X_{i+1}$, and
    \item $K_Z \equiv 0$
\end{itemize}
Then there exists $m=m(d, l)$ such that $mK_Z \sim 0$.
\end{theorem}

\begin{proof}
We set 
$\Phi:= \left\{ \frac 1l, \frac 2l, \dots \frac{l-1}l \right\}$.
Hence for all pairs 
$(X, B) \in \mathfrak D$, 
the coefficients of 
$B$ 
belong to (the finite set)
$\Phi$.
\newline
As $l(K_X+B) \sim 0$, we can write the canonical bundle formula for $(X, B)$ and $f$, cf. Remark~\ref{cbf.lin.eq.rmk}, as
\[
l(K_X+B) \sim f^\ast l(K_Z+B_Z+M_Z).
\]
As $K_Z \equiv 0$, it immediately follows from the definition of $B_Z, M_Z$ that $B_Z = 0 \equiv M_Z$. 
Since $f$ factors as in~\eqref{mfs.tower.eq2}, the general fibre $(F, B|_F)$ is one of the pairs described in Corollary~\ref{logCY.fibred.cor}.
Hence, Corollary~\ref{cbf.triv.cor} implies that there exists $m'=m'(d, \Phi)$ such that $m'M_Z \sim 0$.
Thus,
\[
lm'(K_X+B) \sim 0 \sim f^\ast lm'(K_Z+M_Z) \sim f^\ast lm'K_Z.
\]
To conclude, it suffices to take $m:=lm'$.
\end{proof} 

\subsection{Birational transformations of fibred log Calabi--Yau pairs}
\label{sect.trans.fib.lcy}

In this subsection we collect a few technical results on birational transformations of log Calabi--Yau pairs endowed with a fibration that will be useful in the paper.

\begin{proposition}
\label{small.qfact.lcy.prop}
 Let $(Y, D)$ be a klt pair, $D$ a $\mathbb{Q}$-divisor, and let $f \colon Y \to Z$ be a projective contraction of normal varieties.
 Assume that $K_Y+D \sim_{f, \mathbb{Q}} 0$ and let $s\colon Z' \to Z$ be a small contraction.
Then there exists a 
$\mathbb{Q}$-factorial 
klt pair 
$(Y', D')$ 
isomorphic to 
$(Y, D)$ 
in codimension one and a projective contraction of normal varieties 
$f' \colon Y' \to Z'$.
\end{proposition}

\begin{proof}
Let 
\begin{align}
\label{res.lcy.indet}
	\xymatrix{
Y \ar[d]_{f} & 
\overline{Y} \ar[l]_{p} \ar[d]^{\overline{f}} & 
\\
Z & 
Z' \ar[l]_{s}
			}
\end{align}
be a smooth resolution of indeterminacies of the rational map 
$s^{-1} \circ f \colon Y \dashrightarrow Z'$.
As 
$(Y, D)$ 
is klt, it follows that for 
$0< \delta \ll 1$, 
$K_{\overline{Y}}+ \tilde{D} + (1-\delta) E= p^\ast (K_Y+D) + F$, 
where 
$\tilde{D}$ 
is the strict transform of 
$D$, 
$E$ 
is the support of the 
$p$-exceptional 
divisor and 
$F$ 
is an effective and 
$p$-exceptional 
divisor whose support coincides with that of 
$E$.
Hence, 
$K_{\overline{Y}}+ \tilde{D} + (1-\delta) E 
\sim_{\overline{f}, \mathbb{Q}} F$.
As the diagram in~\eqref{res.lcy.indet} is a resolution of indeterminacies of 
$s^{-1} \circ f$ 
and 
$s$ 
is a small contraction, 
then 
$F \vert_{\overline{Y}_z} \sim 0$
for
$z \in Z' \setminus {\rm exc}(s)$ and
the support of $F$ coincides with the divisorial part of the relative base locus of $K_{\overline{Y}}+ \tilde{D} + (1-\delta) E$ over $Z'$.
By~\cite{HX}*{Theorem 1.1} and~\cite{MR2929730}*{Theorem 1.4} a $(K_{\overline{Y}}+ \tilde{D} + (1-\delta) E)$-good minimal model must exist, as the general fibre of $f$ is a good minimal model for the restriction of $K_{\overline{Y}}+ \tilde{D} + (1-\delta) E$ to a general fibre of $\overline{f}$.
It follows that the relative $(K_{\overline{Y}} +\tilde{D} + (1-\delta) E)$-MMP$/Z'$ contracts $F$ and hence it terminates with a projective contraction $f' \colon Y' \to Z'$ such that $K_{Y'} +D' \sim_{f', \mathbb{Q}} 0$, where $D'$ is the strict transform of $\tilde{D}$.
As this run of the MMP contracts exactly $F$, which is the exceptional locus of $p$, it follows that $Y$ and $Y'$ are isomorphic in codimension one.
\end{proof}

The next result shows that we can modify the base $Z$ of a fibration $f \colon Y \to Z$ by means of an isomorphism $Z \dashrightarrow Z'$  in codimension one and show that $Z'$ is also the base of an elliptic log Calabi--Yau pair. Contrary to the previous statement, here we assume that the base is $\mathbb{Q}$-factorial.

\begin{proposition}
\label{bir.contr.lcy.prop}
Let $(Y, D)$ be a klt pair and let $f \colon Y \to Z$ be a projective contraction of normal varieties.
Assume that $K_Y+D \sim_{\mathbb{Q}} 0$ and $Z$ is $\mathbb{Q}$-factorial and let $t\colon Z \dashrightarrow Z'$ be a birational contraction of normal projective varieties.
Then there exists a $\mathbb{Q}$-factorial klt pair $(Y',D')$ isomorphic to $(Y, D)$ in codimension one and a projective contraction of normal varieties $f' \colon Y' \to Z'$.
\end{proposition}
 
A birational map $t \colon Z \dashrightarrow Z'$ is a birational contraction if $t$ is proper and $t^{-1}$ does not contract any divisors.

\begin{proof}
Let $H'$ be an ample divisor on $Z'$ and let $H$ be its pullback on $Z$. 
$H$ exists and it is well defined as $t$ is a birational contraction.
If we consider the pair $(Y, D+ \epsilon f^\ast H), \; 0<\epsilon\ll 1$ then $f^\ast H$ is abundant, since the Kodaira and numerical dimension are invariant by pullback under contraction morphisms. 
By~\cite{GL}*{Theorem 4.3}, there exists a run of the $(K_Y+D+\epsilon f^\ast H)$-MMP, $Y\dashrightarrow Y''$ which terminates with a good minimal model $Y'' \to Z'$.
Moreover, $Y''$ admits a structure of a log Calabi--Yau pair $(Y'', D'')$, where $D''$ is the strict transform of $D$ on $Y''$.
Let $\{E_1, \dots, E_k\}$ be the divisors contracted by the birational contraction $Y \dashrightarrow Y''$. 
The log discrepancy of any $E_i$ with respect to $(Y, D)$ (or, equivalently, $(Y'', D'')$) is at most $1$.
Hence, by~\cite{BCHM}*{Corollary 1.4.3}, there exists a model $Y' \to Y''$ of $Y''$ on which the only extracted divisors are the $E_i$.
This yields the desired model in the statement of the proposition.
\end{proof}

\section{Rationally connected K-trivial varieties}
\label{rc.cy.sect}

In this section we show that the set of $d$-dimensional rationally connected (in short, RC) klt projective varieties with torsion canonical bundle are bounded up to flops, if we bound the torsion index, i.e., if we assume that there exists a fixed integer $l$ such that $lK_X\sim 0$.
When the dimension of $X$ is 3, this result was implicitly proven in~\cite{CDHJS}*{Theorem 5.1}.

\subsection{Partial resolutions of RC K-trivial varieties and towers structure}
\label{term.rc.cy.sect}
Given a rationally connected klt projective variety $X$ with $K_X \sim_\mathbb{Q}0$, by~\cite{BCHM}*{Corollary 1.4.3}, we can construct a partial resolution $\pi \colon X' \to X$ of $X$ such that
\[
K_{X'}+D = \pi^\ast K_X, \; \lfloor D \rfloor =0,
\]
the divisorial part of the exceptional locus of $\pi$ coincides with the support of $D$, and the pair $(X', D)$ is canonical, $\mathbb{Q}$-factorial with $K_{X'}+D \sim_{\mathbb{Q}} 0$. 
The above conditions imply that $(X', 0)$ is canonical; thus, as $X'$ is rationally connected, $K_{X'}$ cannot be pseudo-effective and we obtain that 
\[
D>0.
\] 
Moreover,~\cite{MR2802603}*{Theorem 2.3} implies that for any $0< \epsilon \ll 1, \; X$ is a good minimal model for the $K_{X'} +(1+\epsilon) D$-MMP.
If we assume that $lK_X \sim 0, \; l \in \mathbb{N}$, then also $l(K_{X'}+D) \sim 0$ and the coefficients of $D$ belong to the subset $\{\frac 1 l, \frac 2 l, \dots, \frac{l-1}{l}\}$.
Under this assumption, the pair $(X', D)$ is a $\frac 1 l$-lc pair.

As $K_{X'}$ is not pseudo-effective, we can apply Theorem~\ref{structure} and there is a crepant birational contraction $(X', D) \dashrightarrow (X'', D')$ to a $d$-dimensional log Calabi--Yau pair $(X'', D')$ such that $l(K_{X''} +D') \sim 0$ and $X''$ is equipped with a non-birational contraction $g \colon Y \to Z$ which can be factored into a sequence of Mori fibrations:
\begin{align}
\label{tower.term.eq}
\xymatrix{
 (X'', D')=(X_0, D_0) \ar[r] & X_1 \ar[r] & \dots \ar[r] & X_{s-1} \ar[r] & X_s=Z.
}
\end{align}

If the pair $(X_0, D_0)$ is not of product type then we know that $\dim Z=0$ and $(X'', D')$ belongs to a log bounded family, by Theorem~\ref{logCY.fibred.cor}.
If $\dim Z > 0$, then $K_Z \equiv 0, \; Z$ is $\mathbb{Q}$-factorial and rationally connected:
Theorem~\ref{mfs.tow.CYend.thm} then implies that there exists $m=m(d, l)$ such that $mK_Z \sim 0$.

\subsection{Boundedness of RC K-trivial varieties with bounded torsion index}
\label{bound.rc.cy.sect}
The decomposition introduced in the previous section, suggest that the following result should hold inductively.

\begin{theorem}
\label{rc.cy.log.bound.thm}
Fix positive integers $d, l$.
Consider varieties $X$ such that
\begin{itemize}
    \item $X$ is klt projective of dimension $d$, 
    \item $X$ is rationally connected, and
    \item $lK_X \sim 0$.
\end{itemize}
Then the set of such $X$ is bounded up to flops.
\end{theorem}

\begin{proof}
We prove the Theorem by induction on $d$.
The case $d=1$ is trivial.
\newline
We will use the notation from \S~\ref{term.rc.cy.sect} and prove the inductive step, assuming the theorem holds in dimension at most $d-1$.
For the reader's convenience, we divide the proof into different steps.

\medskip

{\bf Step 1}.
{\it  
In this step we show that the set of pairs 
$(X'', D')$ 
constructed in~\eqref{tower.term.eq} is log bounded up to flops}.
\newline
It has already been discussed in \S~\ref{term.rc.cy.sect} that the claimed conclusion holds when 
$X''$ 
is not of product type.
Hence we are left to prove the case when in~\eqref{tower.term.eq} 
$\dim Z >0$ 
and 
$K_Z \sim_\mathbb{Q} 0$. 
We have already observed above that there exists 
$m=m(d, l)$ 
such that 
$mK_Z \sim 0$.
In particular, this implies that 
$(Z, 0)$ 
is 
$\frac 1m$-lc.
By the inductive hypothesis,
$Z$ 
is bounded up to flops.
Hence, there exists a klt variety 
$\overline{Z}$ 
isomorphic to 
$Z$ 
in codimension one, 
$mK_{\overline{Z}} \sim 0$, 
and 
$\overline{Z}$ belongs to a bounded family.
Moreover, as 
$Z$ 
and 
$\overline Z$ 
are connected by flops, 
also 
$\overline Z$
is 
$\frac 1m$-lc.
By Remark~\ref{rmk.q-factor.bound}, the set of all 
$\mathbb Q$-factorializations 
of the varieties 
$\overline Z$ 
is also bounded, as noted also right after the statement of~\ref{bir.bound.fibr.cor}.
Thus, we can assume that $\overline{Z}$ is $\mathbb{Q}$-factorial, as well, and by repeatedly applying Proposition~\ref{bir.contr.lcy.prop} we can assume that there exists a tower of Mori fibrations analogous to that in~\eqref{tower.term.eq}
\begin{align*}
\xymatrix{
(\overline{X}'', \overline{D}')=(\overline{X}_0, \overline{D}_0) \ar[r] &
\overline{X}_1 \ar[r] &
\dots \ar[r] &
\overline{X}_{s-1} \ar[r] &
\overline{X}_s=\overline{Z},
 }
\end{align*}
where each $\overline{X}_i$ is $\mathbb{Q}$-factorial and isomorphic in codimension one to the corresponding $X_i$ in~\eqref{tower.term.eq}.
Hence, as $\rho(X_i/X_{i+1})=1$ the same holds for $\rho(\overline{X}_i/\overline{X}_{i+1})$ which implies that $-K_{\overline{X}_i}$ is ample over $\overline{X}_{i+1}$. 
Finally, Theorem~\ref{birk.tower.thm} implies that $(\overline{X}'', \overline{D}')$ is log bounded.

\medskip

{\bf Step 2}. 
{\it 
In this step we show that the set of pairs 
$(X', D)$ is log bounded up to flops}.
\newline
The pairs 
$(X', D)$, 
$(X'', D')$, 
and 
$(\overline{X}'', \overline{D}')$ 
are all crepant birational.
We denote by 
$\{E_1, \dots, E_j\}$ 
the divisors contracted by the birational contraction 
$X' \dashrightarrow X''$.
As 
$X''$, 
$\overline{X}''$ 
are isomorphic in codimension one, the 
$E_i$ 
are exceptional for the rational contraction 
$X' \dashrightarrow \overline{X}''$, 
too.
The log discrepancy of the 
$E_i$ 
with respect to 
$(X', D)$, 
$(X'', D')$, 
and 
$(\overline{X}'', \overline{D}')$ 
is the same and it is contained in 
$(0, 1]$.
By~\cite{BCHM}*{Corollary 1.4.3}, there exists a 
$\mathbb{Q}$-factorial 
klt log Calabi--Yau pair 
$(\overline{X}', \overline{D})$ 
and a morphism 
$\overline{r} \colon \overline{X}' \to \overline{X}''$
extracting exactly the valuations corresponding to 
$\{ E_1, \dots, E_j\}$;
moreover, 
$K_{\overline{X}'}+ \overline{D} 
= \overline{r}^\ast(K_{\overline{X}''}+\overline{D}')$.
By construction, 
$(X', D)$ 
and 
$(\overline{X}', \overline{D})$ 
are isomorphic in codimension one.
Thus, Corollary~\ref{bir.bound.fibr.cor} implies that also the set of pairs 
$(\overline{X}', \overline{D})$ 
that we have constructed is log bounded which terminates the proof of this step since 
$(\overline{X}', \overline{D})$ 
and 
$(X', D)$
are isomorphic in codimension one.

\medskip

{\bf Step 3}. 
{\it In this step we show that $X$ is bounded up to flops}.
\newline
We know that the pair $(\overline{X}', \overline{D})$ is log bounded and it is isomorphic in codimension one to $(X', D)$.
Moreover, the Kodaira dimension of $K_{\overline{X}'}+ \lceil\overline{D}\rceil$ is $0$ as $K_X \sim_\mathbb{Q} 0$; 
thus, a minimal model for the $(K_{\overline{X}'}+(1+\epsilon)\overline{D})$-MMP on $\overline{X}'$, for $\epsilon \ll 1$, will be a $\mathbb{Q}$-factorial variety $\overline{X}$ with $lK_{\overline{X}} \sim 0$.
By construction, $X$ and $\overline{X}$ are isomorphic in codimension one. 
If we can prove that $\overline{X}$ belongs to a bounded family then the theorem follows at once.
But this is just a consequence of Proposition~\ref{flops.lcy.family.prop} below.
\end{proof}

\begin{proposition}
\label{flops.lcy.family.prop}
Fix a DCC set $\mathfrak{R} \subset (0, 1) \cap \bQ$.
Let $\mathfrak{D}$ be a bounded set of klt pairs.
For any pair $(X, B) \in \mathfrak{D}$ assume that: 
\begin{itemize}
\item 
$B > 0$ and its coefficients belong to $\mathfrak{R}$,

\item 
$K_X+B \sim_\bQ 0$,

\item 
the Kodaira dimension of $-K_X$ is zero, and

\item 
there exists a good minimal model $X_1$ for $K_X+(1+\epsilon) B$, $0< \epsilon \ll 1$.
\end{itemize}   
Then the set 
\begin{align*}
\mathfrak{D}_1:=
\{ &
X_1 \ \vert
\ (X, B) \in \mathfrak D, 
X_1 \textrm{ is a good minimal model}
\\ 
&
\textrm{for } (K_X+(1+\epsilon)B), \; 
\textrm{for} \; 
0< \epsilon \ll 1
\}
\end{align*}
is bounded up to flops.
\end{proposition}

In the hypotheses of Proposition~\ref{flops.lcy.family.prop} the minimal model $X_1$ is a Calabi--Yau variety, $K_{X_1} \sim_\bQ 0$, which is klt but non-canonical.
Moreover, under these assumptions, the minimal model $X_1$ is independent of the choice of $1< \epsilon \ll 1$.

\begin{proof}
As $\mathfrak{R}$ is a DCC set, by~\cite{MR3224718}*{Theorem 1.5} there exists a finite subset $\mathfrak{R}_0 \subset \mathfrak{R}$ such that the coefficients of pairs in $\mathfrak{D}$ are in $\mathfrak{R}_0$.
Moreover,~\cite{DCS}*{Corollary 2.9} implies that there exists $\epsilon >0$ such that all pairs in $\mathfrak{D}$ are $\epsilon$-klt.
By Proposition~\ref{fam.good.prop}, there exists a pair $(\mathcal{X}, \mathcal{B})$ together with a projective morphism of quasi-projective varieties $\mathcal{X} \to S$ such that for any pair $(X, B) \in \mathfrak{D}$ there exists $s \in S$ and an isomorphism $h_s \colon \mathcal{X}_s \to X$ and $\mathcal{B}|_{\mathcal{Z}_s} = h^\ast_s B$.
By taking the Zariski closure of the points $s \in S$ corresponding to pairs in $\mathfrak{D}$, and decomposing that into a disjoint union of finitely many locally closed subsets, we can assume that $S$ is smooth and the points on $S$ corresponding to pairs in $\mathfrak{D}$ are dense.
By~\cite{MR3507257}*{Proposition 2.4} up to substituting $S$ with a Zariski dense open, we can assume that $(\mathcal{X}, \mathcal{B})$ is klt and $\mathbb{Q}$-factorial.
Restricting to a Zariski open set of $S$ and/or decomposing $S$ into a disjoint union of finitely many locally closed subsets are operations that we are allowed to perform: by noetherian induction, these operations can be performed only finitely many times, hence they do not affect any boundedness argument.
\newline
By passing to a sufficiently high log resolution 
\[
\xymatrix{
\mathcal{Y} \ar[rr]^\sigma \ar[dr] & &\mathcal{X} \ar[dl]\\
& S &
}
\]
of $(\mathcal{X}, \mathcal{B})$, there exist effective divisors $\mathcal E, \mathcal F$ on $\mathcal Y$ such that $K_\mathcal{Y} + \mathcal{E}=\sigma^\ast(K_\mathcal{X}+ \mathcal{B}) + \mathcal{F}$, with $\sigma_\ast \mathcal E= \mathcal B$ and $\mathcal{F}$ is $\sigma$-exceptional. 
Up to decomposing $S$ in a disjoint union of finitely many locally closed subsets and possibly passing to a higher model $\mathcal{Y}$, we can assume that $(\mathcal{Y}, \mathcal{E})$ is log smooth over $S$.
Moreover, we can also assume that the support of $\mathcal E$ contains all $\sigma$-exceptional divisors and that $\lfloor \mathcal E \rfloor =0$.
\newline
Let us notice that for any sufficiently divisible $m \in \mathbb{N}$, and for $0<\epsilon \ll 1, \; \epsilon \in \mathbb{Q}$, 
\[ 
h^0(\mathcal{Y}_s, \mathcal{O}_{\mathcal{Y}_s}(m(K_{\mathcal{Y}_s} +(1+\epsilon) \mathcal{E}|_{\mathcal{Y}_s}))) =  h^0(\mathcal{X}_s, \mathcal{O}_{\mathcal{X}_s}(m(K_{\mathcal{X}_s} +(1+\epsilon) \mathcal{B}|_{\mathcal{X}_s}))).
\]
By~\cite{MR3779687}*{Theorem 4.2}, 
\[ 
h^0(\mathcal{Y}_s, \mathcal{O}_{\mathcal{Y}_s}(m(K_{\mathcal{Y}_s} +(1+\epsilon) \mathcal{E}|_{\mathcal{Y}_s})))
\] 
is a constant function of $s \in S$; 
in particular $\kappa(\mathcal{Y}_s, \mathcal{O}_{\mathcal{Y}_s}(K_{\mathcal{Y}_s} +(1+\epsilon) \mathcal{E}|_{\mathcal{Y}_s})) = 0, \; \forall s\in S$.
Thus, for any $s \in S, \; K_{\mathcal{Y}_s} +(1+\epsilon) \mathcal{E}|_{\mathcal{Y}_s}$ is not movable.
Then~\cite{MR3779687}*{Theorem 1.2} implies that for any $s \in S$ a good minimal model for $K_{\mathcal{Y}_s} +(1+\epsilon) \mathcal{E}|_{\mathcal{Y}_s}$ exists.
If we run the $K_{\mathcal{Y}} +(1+\epsilon)\mathcal{E}$-MMP over $S$,~\cite{MR3779687}*{Theorem 1.2} implies that a good minimal $\mathcal{Z} \to S$ model must exist and by~\cite{MR3779687}*{Lemma 6.1} $\mathcal{Z}_s$ will yield a good minimal model for each pair $(\mathcal{Y}_s, (1+\epsilon) \mathcal{E}|_{\mathcal{Y}_s})$, thus concluding the proof.
As $\mathcal{Z}_s$ is a good minimal model for $K_{\mathcal{Y}_s} + (1+\epsilon) \mathcal{E}|_{\mathcal{Y}_s}$, it follows from the definition of good minimal model that $X_1$ and $\mathcal{Z}_s$ are isomorphic in codimension one.
\end{proof}

The following immediate corollary of the Theorem~\ref{rc.cy.log.bound.thm} extends~\cite{CDHJS}*{Cor. 5.2} to any dimension. 
It relates the boundedness of rationally connected K-trivial klt varieties to the boundedness of the torsion index of the canonical divisor inside the class group.

\begin{corollary}
\label{equiv.rc.cy.bound.cor}
Fix a positive integer $d$.
Let $\mathfrak{C}$ be the set of varieties $X$ satisfying the following hypotheses:
\begin{enumerate}
    \item $X$ is a klt projective variety of dimension $d$, 
    \item $X$ is rationally connected, and
    \item $K_X \equiv 0$.
\end{enumerate}
Then, $\mathfrak{C}$ is bounded up to flops if and only if there exists a positive integer $l=l(d)$ such that $lK_X \sim 0$ for any $X \in \mathfrak{C}$.
\end{corollary}

\begin{proof}
If there exists $l=l(d)$ such that $lK_X \sim 0$ for any $X \in \mathfrak{C}$, then the boundedness up to flops of $\mathfrak{C}$ is a consequence of Theorem~\ref{rc.cy.log.bound.thm}.
\newline
Let us assume that $\mathfrak{C}$ is bounded up to flops.
Hence, there exists  $h\colon \mathcal{Z}\rightarrow S$ a projective morphism of schemes of finite type such that for each $X\in \mathfrak{C}$ is isomoprhic in codimension one to $\mathcal{Z}_s$ for some closed point $s\in S$ and, moreover, $\mathcal{Z}_s$ is normal.
As $lK_{\mathcal{Z}_s} \sim 0$ if and only if $lK_X \sim 0$, it suffices to show that there exists a positive integer $l$ such that $lK_{\mathcal{Z}_s} \sim 0$ for any $s \in S$.
Moreover, by~\cite{Bir16a}*{Lemma 2.25} we can assume that there exists $I=I(\mathfrak C)$ such that $IK_{\mathcal{Z}_s}$ is Cartier for any $s \in S$.
\newline
Decomposing $S$ into a finite union of locally closed subsets and possibly discarding some components, we may assume that $S$ is smooth and that every fibre $\mathcal{Z}_s$ is a normal variety.
Up to shrinking $S$, we can also assume that the set 
\begin{align*}
    S' := 
    \left\{
    s\in S \ \vert \ 
    \exists X \in \mathfrak{C} 
    \textrm{ such that } X \textrm{ is isomorphic in codimension one to } 
    \mathcal{Z}_s
    \right\}
\end{align*}
is Zariski dense in $S$.
Furthermore, by~\cite{HX}*{Proposition 2.4}, decomposing $S$ into a finite union of locally closed subsets and discarding those components that do not contain points of $S'$, we can assume that $K_\mathcal{Z}$ is $\mathbb{Q}$-Cartier.

\medskip

{\bf Claim}. 
{\it Up to decomposing $S$ into a finite union of locally closed subsets, we may assume that for any  connected component $\bar{S}$ of $S$, $h^0(\mathcal{Z}_s, \mathcal{O}_{\mathcal{Z}_s}(mI(K_{\mathcal{Z}_s})))$ is independent of $s \in \bar{S}$, for all $m >0$}.
\begin{proof}
Up to decomposing $S$ into a finite union of locally closed subsets, we may assume that there exists a log resolution $\psi\colon (\mathcal{Z}', \mathcal{B}') \to \mathcal Z$ of $(\mathcal Z, 0)$, where $\mathcal{B}'$ is the exceptional divisor of $\psi$.
Furthermore, we may also assume that for any $s \in S$, $(\mathcal{Z}'_s, \mathcal{B}'_s)$ is a log resolution of $\mathcal{Z}_s$.
In particular, for any $s \in S$, for all $m > 0$
\begin{align*}
H^0(\mathcal{Z}_s', \mathcal{O}_{\mathcal{Z}_s'}(m(K_{\mathcal{Z}_s'}+\mathcal{B}_s'))) = 
H^0(\mathcal{Z}_s, \mathcal{O}_{\mathcal{Z}_s}(mK_{\mathcal{Z}_s})).
\end{align*}
The conclusion then follows from~\cite{MR3779687}*{Theorem 4.2}.
\end{proof}
\noindent
At this point, we discard those connected components of $S$ that do not contain points of $S'$.
By construction then, given a connected component $S_i$ of $S$, there is a point $s \in S'\cap S_i$; if $m_i$ is a positive integer such that $m_iK_{\mathcal{Z}_s} \sim 0$, the claim above implies that 
\begin{align}
\label{inv.plurig.cy.eqn}
h^0(\mathcal{Z}_s, \mathcal{O}_{\mathcal{Z}_s}(m_iK_{\mathcal{Z}_s}))=1, \textrm{ for all $s \in S_i$.}
\end{align} 
We define $l$ to be the maximum of the positive integers $m_i$ just defined.
This is well defined since $S$ has only finitely many connected components, being of finite type.
As for any $s \in S'$, $K_{\mathcal{Z}_s} \sim_{\mathbb{Q}} 0$, then~\eqref{inv.plurig.cy.eqn} implies that $lK_{\mathcal{Z}_s} \sim 0$.
\end{proof}

\section{Elliptic Calabi--Yau varieties with a rational section}
\label{ell.cy.sect}

In this section we will prove some results regarding the geometric structure of elliptic Calabi--Yau varieties and of the bases of such fibrations that will be needed in the proof of the main theorems of this paper.

\subsection{Birational geometry of bases of fibred Calabi--Yau varieties} \label{base.simply.cy.sect}

In this subsection we prove that the base of a Calabi--Yau manifold endowed with a non-birational fibration is rationally connected.

Koll\'ar and Larsen,~\cite{MR2641190}*{Theorem 3}, proved that a simply connected  smooth (or canonical) K-trivial variety $Y$ endowed with a dominant rational map $m \colon Y \dashrightarrow X$ to a non-uniruled variety $X$ is isomorphic to a product $Y \simeq Y_1 \times Y_2$, where the map $m$ restricted to $Y_2$ induces a dominant generically finite map 
$Y_2 \dashrightarrow X$.
In the case of a fibred Calabi--Yau manifold, such a product cannot exist, as we now show.
We remind the reader that in this paper a Calabi--Yau manifold 
$Y$ 
is assumed to be simply connected and to have vanishing Hodge numbers 
$h^{i, 0}$ 
for 
$1 \leq i \leq \dim X -1$.

\begin{corollary}\label{rc.base.cor}
Let $Y$ be a projective Calabi--Yau manifold. 
Assume that $Y$ is endowed with a morphism $f \colon Y \to X$ of relative dimension $0< d < \dim Y$.
Then $X$ is rationally connected.
\end{corollary}

\begin{proof}
Let us assume that 
$X$ 
is not rationally connected. 
Considering the MRC fibration of 
$X$, 
$X \dashrightarrow W$, 
its image 
$W$ 
is a non-uniruled variety of positive dimension by~\cite{GHS}*{Corollary~1.4}.
Hence, by the results of Koll\'ar and Larsen 
$Y$ 
would have to be isomorphic to a product 
$Y \simeq Y_1 \times Y_2$, 
with 
$\dim Y_1, \dim Y_2 >0$. 
As 
$K_Y\sim 0$, 
K\"unneth's formula for Hodge numbers implies that 
$h^0(K_{Y_1})=1=h^0(K_{Y_2})$.
As 
$h^0(K_{Y_1})=h^{\dim Y_1, 0}(Y_1)$, 
pulling back from 
$Y_1$ 
to 
$Y$, 
we get that 
$h^{\dim Y_1, 0}(Y) \neq 0$, 
which contradicts 
$Y$ 
being Calabi-Yau, since 
$\dim Y_1 < \dim Y$.
\end{proof}

Slightly more general versions of this result have appeared also in~\cite{1808.01115}*{Theorem~1.1} and~\cite{Lin20}*{Theorem~1.4}.

\subsection{Boundedness for elliptic Calabi--Yau varieties over bounded bases}
The main goal of this subsection is to prove a generalization of~\cite{DCS}*{Theorem~1.3}.
In what follows, by a rational section $X \dashrightarrow Y$ of a contraction $f\colon Y \to X$, we will mean a rational map $s \colon X \dashrightarrow Y$ such that $f \circ s$ is the identity on $X$.
With this notation, we will denote by $S$ the Zariski closure of the image of $X$ via $s$ and by $S^\nu$ the normalization $\nu \colon S^\nu \to S$.

\begin{theorem}
\label{ell.cy.bound.base.thm}
Fix positive integers $n, d, l$.
Then the set of varieties $Y$ such that
\begin{enumerate}
    \item $Y$ is a proper klt variety of dimension $n$,
    \item $lK_Y \sim 0$,
    \item $Y \to X$ is an elliptic fibration with a rational section $X \dashrightarrow Y$, 
    \item there exists a very ample Cartier divisor $A$ on $X$ with $A^{n-1} \leq d$
\end{enumerate}
is bounded up to isomorphism in codimension one.

If, moreover, in (1) we further assume that 
$Y$ 
is projective, 
then the set of varieties satisfying the conditions above is bounded up to flops.
\end{theorem}

The existence of a very ample Cartier divisor $A$ with bounded volume implies that the set of bases of the elliptic fibrations is bounded.

\begin{proof}
For the reader's convenience, we divide the proof in several steps. 

We first show that any proper variety 
$Y$ 
satisfying conditions 
(1)-(4)
in the statement of the theorem is isomorphic in codimension one to a projective klt variety 
$Y'$ 
that also satisfies properties 
(2)-(3).
That implies that it suffices to consider the case where $Y$ is projective in the proof of the theorem.
To show the above statement, let us consider a log resolution
$r \colon \widetilde Y \to Y$ 
of 
$Y$
with 
$\widetilde Y$ 
projective, which exists by Hironaka's resolution of singularities and Chow lemma.
As 
$Y$, 
by assumption is endowed with an elliptic fibration with a rational section over 
$X$, 
then the same holds for 
$\widetilde Y$.
Writing 
$K_{\widetilde Y}+\widetilde B=r^\ast(K_Y)+G$, 
then 
$widetilde B$, $G$
are effective 
(but possibly $0$) 
with no common components and 
$\lfloor \widetilde B \rfloor=0$.
Moreover, since 
$lK_Y \sim 0$
then for all 
$0< \epsilon \ll 1$, 
$K_{\widetilde Y}+(1+\epsilon)\widetilde B 
\sim_{\mathbb R} 
\epsilon \widetilde B + G$.
As 
$ \epsilon \widetilde B + G$
is 
$r$-exceptional, 
then 
$K_{\widetilde Y}+(1+\epsilon)\widetilde B$
is linearly trivial along the general fiber of the morphism 
$\widetilde Y \to X$.
Hence, by~
there exists a projective variety 
$Y'$
that will be a
$(K_{\widetilde Y}+(1+\epsilon)\widetilde B)$ good minimal model over 
$X$.
Moreover, since 
$ \epsilon \widetilde B + G$
is very exceptional for 
$\widetilde Y \to X$ 
by construction of 
$\widetilde Y$,
cf. Definition~\ref{deg.div.def},
it then follows from 
Lemma~\ref{deg.div.cy.lem} 
that the support of 
$ \epsilon \widetilde B + G$
is exceptional for the birational contraction (over $X$)
$\widetilde Y \dashrightarrow Y'$.
Hence, 
$K_{Y'} \sim_{\mathbb Q} 0$
and the condition 
$lK_Y \sim 0$
and the above observations immediately imply that
$lK_{Y'} \sim 0$.
The fact that $Y'$ satisfies the other conditions in the statement of the theorem immediately follow from our construction.

\medskip

{\bf Step $0$}. 
{\it In this step we show that there exists a $\mathbb{Q}$-factorialization $X'$ of $X$ which belongs to a bounded family and we construct an elliptic fibration $Y' \to X'$ with $Y'$ $\mathbb{Q}$-factorial and isomorphic to $Y$ in codimension one.}
\newline
As $X$ is the base of an elliptic fibration and $Y$ is klt, then $X$ supports an effective divisor $\Gamma$ such that $(X, \Gamma)$ is klt, see Theorem~\ref{bundle}.
By~\cite{MR3224718}*{Theorem~1.1} and~\cite{MR2448282}*{Theorem~8.1}, it is possible to choose $\Gamma$ so that its coefficients vary in a DCC set $I=I(n)$.
Hence, by~\cite{Bir16a}*{Lemma~2.48} it follows that there exists $\epsilon=\epsilon(n, I)$ such that $(X, \Gamma)$ is $\epsilon$-lc.
As the DCC set defined $I$ above only depends on $n$, and $\epsilon$ is a function of $n, I$, we can conclude that $\epsilon$ only depends on $n$.
\newline
Hence, by~\cite{BCHM}*{Corollary 1.4.3} $X$ admits a small $\mathbb{Q}$-factorialization $X' \to X$. 
$X'$ belongs to a bounded family, by Remark~\ref{rmk.q-factor.bound}.
In particular, there exists an integer $d'=d'(n, d, l)$ and a very ample Cartier divisor $A'$ on $X'$ such that $A'^{n-1} \leq d'$.
Applying Proposition~\ref{small.qfact.lcy.prop} with $D=0$, we see that there exists a $\mathbb{Q}$-factorial K-trivial klt variety $Y'$ isomorphic to $Y$ in codimension one and an elliptic fibration $f' \colon Y' \to X'$.
It is immediate to see that the strict transform $S'$ of $S$ on $Y'$ is a rational section of $f'$.
\\
We will denote by $A''$ the divisor $A'':=(4n+4)A'$. 

\medskip

{\bf Step $1$}. 
{\it In this step, we show that it suffices to prove that the set of varieties
$Y'$ 
constructed in Step 
$0$ 
is bounded up to flops}.
\newline
As 
$Y$ 
and 
$Y'$ 
are K-trivial varieties that are klt and isomorphic in codimension one,
the claim is immediately verified. 

\medskip

{\bf Step $2$}. {\it In this step, we show that there exists a log bounded set of morphisms $\mathfrak F$ and a 4-tuple $(f_{Y''}, Y'', 0, X') \in \mathfrak F$ together with a birational contraction $Y'' \dashrightarrow Y'$ over $X'$ that is crepant with respect to $(Y', 0)$.}

\medskip

The proof of Step 2 is divided into several steps (Step 2.1-2.6).
We start by introducing some notation.

\medskip

Let $(Y^t, \Delta^t)$ be a terminalization $t \colon Y^t \to Y'$ of $(Y', 0)$, $K_{Y^t}+\Delta^t = t^\ast K_{Y'} \sim_\mathbb{Q} 0$.
We denote with $S^t$ the strict transform of $S'$ on $Y^t$.
As $Y'$ is klt and $lK_{Y'} \sim 0$, then $\Delta^t \geq 0$ and its positive coefficients belong to the set $\{\frac 1 l, \frac 2 l, \dots, \frac{l-1}{l}\}$.
It follows from Definition~\ref{deg.div.def} that the exceptional divisor $E$ of $t$ is very exceptional for the composition $f' \circ t$, cf. Remark~\ref{q-fact.degen.rmk}. 
The support of $E$ consists of the support of $\Delta$ together with divisors of log discrepancy one for $(Y, 0)$.

\medskip

{\bf Step 2.1} { \it In this step, we show that there exists a $\bQ$-factorial projective variety $Z$ which is isomorphic in codimension one to $Y^t$, such that $(Z, 0)$ is terminal and the strict transform $S_Z$ of $S^t$ is big and nef over $X'$}.
\newline
Fix $0 < \epsilon \ll 1$ and run the $(K_{Y^t} + \Delta^t + \epsilon S^t)$-MMP over $X'$. 
As $(Y^t, \Delta^t + \epsilon S^t)$ is klt and $K_{Y^t}+ \Delta^t + \epsilon S^t \sim_\mathbb{R}\epsilon S^t$ is relatively big over $X'$, this run of the MMP must terminate with a minimal model $f_Z \colon Z \to X'$ for $K_{Y^t} + \Delta^t + \epsilon S^t$.
Denoting by $\Delta_Z$ (resp. $S_Z$) the strict transform on $Z$ of $\Delta^t$ (resp. $S^t$), it follows that $S_Z$ is nef and big over $X'$, since $K_{Y^t}+\Delta^t \sim_\bQ 0$ and $S^t$ is horizontal over $X'$.
Moreover, as $(Y^t, \Delta^t)$ is terminal and $K_{Y^t} + \Delta^t + \epsilon S^t \equiv \epsilon S^t$, it follows that every step of this MMP is a $(K_{Y^t}+\Delta^t)$-flop and $Z$ will also be terminal as no divisor is contracted.

\medskip

{\bf Step 2.2}. 
{\it In this step, we show that we can run the relative 
$(K_Z+S_Z)$-MMP 
over 
$X'$ 
with scaling of an ample divisor which terminates with a minimal model 
$f_{Z_1} \colon (Z_1, S_{Z_1}) \to X'$. 
We show that each step of this MMP is a 
$(K_{Z} +\Delta_Z)$-flop 
and that 
$(Z_1, \Delta_{1})$ is terminal.
Here 
$\Delta_{1}$ 
denotes the strict transform of 
$\Delta_Z$ 
on 
$Z_1$.
Moreover, 
$(Z_1, S_{Z_1})$ 
is plt, where 
$S_{Z_1}$ 
denotes the strict transform of 
$S_Z$ 
on 
$Z_1$
}.
\newline
As 
$S_Z$ 
is big and nef over 
$X'$, 
then it is relatively semiample, by the Base point free theorem~\cite{KM}*{Theorem 3.3}, as 
$\mu S_Z \sim_\mathbb{R} K_{Z} +\Delta_Z +\mu S_Z$.
Thus, we can choose 
$\bar{S} \sim_{f_Z, \bQ} S_Z$ 
such that
$\lfloor \bar{S} \rfloor =0$
and
$(Z, \bar{S})$ 
is terminal, by Bertini's theorem, since 
$(Z, 0)$ 
is terminal.
Thus, we can run the 
relative 
$(K_Z + \bar{S})$-MMP 
over 
$X'$ 
with scaling of an ample divisor which must terminate with a 
$\mathbb{Q}$-factorial 
minimal model 
$f_{Z_1} \colon Z_1 \to X'$, 
since 
$\deg((K_Z+\bar{S})\vert_F)=1$ 
along a general fibre 
$F$ 
of 
$f_Z$.
\newline 
We denote by 
$S_{Z_1}$ 
(resp. 
$\Delta_{Z_1}$) 
the strict transform of 
$S_Z$ 
(resp. 
$\Delta^t$) 
on 
$Z_1$.
It follows that $K_{Z_1}+S_{Z_1}$ is also big and nef over $X'$, as $S_Z \sim_\mathbb{Q} \bar{S}$.
As the pair $(Z, \bar{S})$ is terminal and $Z_1$ is a minimal model for the $(K_Z+\bar{S})$-MMP, then $(Z_1, 0)$ is terminal as well.
\newline
We denote by $\nu_{S_{Z_1}} \colon S_{Z_1}^\nu \to S_{Z_1}$ the normalization of $S_{Z_1}$.
By Lemma~\ref{section.lemma}, as $S_{Z_1}$ is a rational section for $f_{Z_1}$ and $Z_1$ is terminal, it follows that $(f_{Z_1} \circ \nu_{S_{Z_1}})_\ast \diff(0)=0$.
The negativity lemma then implies that 
\begin{equation}
\label{log.pullback.plt.eqn2}
    K_{S_{Z_1}^\nu}+\diff(0)=(f_{Z_1} \circ \nu_{S_{Z_1}})^\ast(K_{X'}) -E,
\end{equation}
where $E$ is an effective divisor exceptional over $X'$.
Hence, the pair $(S_{Z_1}^\nu, \diff(0))$ is klt, since $(X', 0)$ is. 
By inversion of adjunction, 
$S_{Z_1}$ 
is the only lc center of the pair 
$(Z_1, S_{Z_1})$
which is then plt, 
which in turn implies that 
$S_{Z_1}$ 
is normal. 

\medskip

{\bf Step 2.3}.
{\it Let $G_1 \in |f_{Z_1}^\ast(4n+4) A'|$ be a general member. In this step, we show that for $\lambda \in (0, 1)$, $(Z_1, \lambda (S_{Z_1}+G_1))$ is $(1-\lambda)$-klt and $K_{Z_1}+S_{Z_1}+G_1$ is big and nef}.
\newline
$(Z_1, S_{Z_1}+G_1)$ is dlt: in fact, $(Z_1, S_{Z_1})$ is plt, and $G_1$ is a general element of a base point free linear system; thus we can conclude by Bertini's Theorem for discrepancies. 
As discrepancies are linear functions of the divisorial part of a pair, and $(Z_1, 0)$ is terminal, the pair $(Z_1, \lambda (S_{Z_1}+G_1))$ is $(1-\lambda)$-klt. 
As $K_{Z_1}+S_{Z_1}$ is nef and big/$X'$ and $G_1 \sim f_{Z_1}^\ast A''$ where $A''$ is Cartier and very ample, the cone theorem implies that $K_{Z_1}+S_{Z_1}+G_1$ is nef and big. 

\medskip

{\bf Step 2.4}. 
{\it In this step, we show that there exists $C'=C'(n, d, l)$ such that $(K_{Z_1}+ S_{Z_1}+(n-2)G_1)^n \leq C'$}.
\newline
Let us denote by 
\[
\alpha = G_1, \; \beta=K_{Z_1}+S_{Z_1}+(n-2)G_1. 
\]
Then, 
\[
(\alpha^{n-2} \beta^2)^{n-1} \geq (\alpha^{n-1} \beta)^{n-2} \beta^{n},
\]
see~\cite{MR2095471}*{Theorem 1.6.3.(i)}.
Hence, to prove that the statement of Step 2.4 holds it suffices to show that there exists positive constants $C_1=C_1(n, d, l)$, $C_2=C_2(n, d, l)$ such that 
\begin{align*}
(\alpha^{n-2} \beta^2)^{n-1} \leq C_1 \ 
\text{and} \
(\alpha^{n-1} \beta)^{n-2} \geq C_2.
\end{align*}
To compute $C_2$, let us note that $0< \alpha^{n-1} \beta$ since $0 \neq G_1^{n-1}$ is a movable class while $K_{Z_1}+S_{Z_1}+G_1$ is big and nef, cf.~\cite{MR3019449}*{Theorem~0.2}.
Moreover,
\begin{align}
\label{eqn.alpha.beta.1}
\nonumber
0< \alpha^{n-1} \beta &= \\ 
\nonumber
G_1^{n-1} \cdot (K_{Z_1}+S_{Z_1}+(n-2) G_1) &= \\
G_1^{n-1} \cdot (K_{Z_1}+S_{Z_1}) &= & 
\hfill[\text{since } G_1^{n}=0]\\
\nonumber
G_1^{n-1} \cdot S_{Z_1} - G_1^{n-1} \cdot \Delta_{Z_1}. & &
\hfill[\text{since } K_{Z_1} \equiv -\Delta_{Z_1}]
\end{align}
As $G_1$ is Cartier and $G_1 \sim f^\ast_{Z_1} (4n+4)A'$,
\begin{align}
\label{eqn.alpha.beta.2}
G_1^{n-1} \cdot S_{Z_1} = 
(4n+4)^{n-1}A'^{n-1} \in \mathbb{N}_{>0},
\end{align}
where we used the fact that $S_{Z_1} \to X'$ is a birational map and $A'$ is very ample and Cartier on $X'$.
As the coefficients of $\Delta_{Z_1}$ belong to the set $\{ \frac 1 l, \frac 2 l, \dots \frac{l-1}{l}\}$, we can write $\Delta_{Z_1}$ as the sum of its prime components
$\Delta_{Z_1} = \sum_{j=1}^s c_j D_j$, 
where 
$c_j \in \left\{ \frac 1 l, \frac 2 l, \dots \frac{l-1}{l}\right\}$, 
and 
$D_j$ 
is a prime divisor for
$j=1, \dots, s$.
\newline
As $G_1$ is Cartier and nef, then
\begin{align}
\label{eqn.alpha.beta.3}
G_1^{n-1} \cdot \Delta_{Z_1} = 
\sum_{j=1}^s c_j (G_1^{n-1} \cdot D_j) 
\quad
\text{ and } 
\quad
G_1^{n-1} \cdot D_j \in \mathbb{N}.
\end{align}
By putting together~\eqref{eqn.alpha.beta.1}-\eqref{eqn.alpha.beta.3}, we obtain that 
\begin{align}
\label{eqn.alpha.beta.4}
\nonumber
0< \alpha^{n-1} \beta &= \\
G_1^{n-1} \cdot S_{Z_1} &- G_1^{n-1} \cdot \Delta_{Z_1} \ \in \frac 1 l \mathbb{N}_{>0}.
\end{align}
Hence, it suffices to take 
$C_2(n, d, l):=\frac 1 l$.
\newline
We now show the existence of the constant $C_1=C_1(n, d ,l)$ such that 
\begin{align*}
(\alpha^{n-2} \beta^2)^{n-1} \leq C_1.
\end{align*}
Let $\bar{A} \subset X'$ be a sufficiently general curve which is a complete intersection of $n-2$ elements of $|A''|$.
In particular, $\bar{A}$ belongs to a bounded family.
\newline
Since $X'$ is normal, $\bar{A}$ is a smooth curve and by the adjunction formula
\begin{align}
\nonumber
\deg_{\bar{A}} K_{\bar{A}} = & (K_{X'} + (n-2) A'') \cdot (A'')^{n-2} & \\
\label{adj.eqn1}
\leq & (n-2)A''^{n-1} & [\text{as } K_{X'} \leq  0 \text{ and } A'' \text{ is ample}]\\ 
\nonumber
\leq & d'(n-2) (4n+4)^{n-1} & [A''=(4n+4)A', \  {\rm Vol}(A')\leq d'].
\end{align}
Let us define $T := Z_1 \times_{X'} \bar{A}$ and $S_T := S_{Z_{1}}|_{T}$.
As $T$ is the complete intersection of $n-2$ sufficiently general divisors in a base point free linear system and $(Z_1, 0)$ is terminal, Bertini's theorem and adjunction imply that $T$ is smooth and $(T, S_T)$ is plt.
Moreover, $T$ is an elliptic surface over $\bar A$, by construction admitting the section $S_T$.
We will denote by $F$ the numerical class of a fibre of the morphism $f_T:= f_{Z_1}\vert_T \colon T \to \bar{A}$.
With this notation, then
\begin{equation}
\label{alpha.beta.eqn1}
\alpha^{n-2} \beta^2 = (K_{Z_1}+S_{Z_1}+(n-2)G_1)^2 \cdot T = ((K_{Z_1}+S_{Z_1}+(n-2)G_1)|_T)^2.
\end{equation}
By the adjunction formula, 
\begin{align}
\label{eqn.adj.formula.T}
 K_T=& (K_{Z_1}+(n-2)G_1)|_T.
\end{align}
Hence,
\begin{align}
\nonumber
 K_T=  
 & 
 (K_{Z_1}+(n-2)G_1) \cdot  T 
 &  
[\text{by } \eqref{eqn.adj.formula.T}]\\
\nonumber  
= 
& 
K_{Z_1} \cdot T +(n-2)(G_1 \cdot  T) 
&  
\\
\label{adj.eqn.2} 
= & 
K_{Z_1} \cdot G_1^{n-2} +(n-2)(G_1 \cdot  T) &  
[T=f_{Z_1}^\ast \bar{A}\equiv f_{Z_1}^\ast A''^{n-2} , \ G_1 \sim f_{Z_1}^\ast A''] \\
\nonumber
\equiv & 
-\Delta_{Z_1} \cdot G_1^{n-2} + (n-2)(A''^{n-1}) F & 
[K_{Z_1} + \Delta_{Z_1} \sim_\bQ 0, \ G_1|_{T} \equiv (A''^{n-1}) F]\\
\nonumber 
\leq & 
 d'(n-2)(4n+4)^{n-1}F & 
[-\Delta_{Z_1} \cdot G_1^{n-2} \leq 0 \ \text{and}\\
\nonumber && A''^{n-1} \leq d'(4n+4)^{n-1}]
\end{align}
Thus, we can rewrite~\eqref{alpha.beta.eqn1}, as
\[
\alpha^{n-2} \beta^2 =  
(K_{T}+S_{T})^2 \leq (K_{T}+S_{T}) \cdot (d'(4n+4)^{n-1}(n-2)F + S_T).
\]
As the general fibre of $T \to \bar{A}$ has genus $1$ and $S_T$ is a section, $(K_{T}+S_{T}) \cdot F = 1$.
On the other hand, the adjunction formula implies that
\begin{align*}
& (K_{T}+S_{T}) \cdot S_T =
 \deg_{S_T} K_{S_T} = 
 \deg_{\bar{A}} K_{\bar{A}} \leq 
d'(n-2)(4n+4)^{n-1}
\end{align*}
since $S_T \to \bar{A}$ is an isomorphism; the displayed inequality is just~\eqref{adj.eqn1}.
Hence, taking $C_1 = 2d'(n-2)(4n+4)^{n-1}$ proves the claim. 

\medskip

{\bf Step 2.5}. 
{\it 
In this step, we show that there exists a log bounded set of morphisms 
$\mathfrak{F}_1$ 
and a 4-tuple 
$(f_{Z'}, Z', G' +\Delta', X') 
\in \mathfrak{F}_1$ 
such that 
$Z'$ 
is 
$\mathbb{Q}$-factorial 
and isomorphic in codimension one to the variety 
$Z$ 
constructed in Step 2.1}.
\newline
We know that 
$(Z_1, 0)$ 
is terminal, 
$(Z_1, \Delta_1)$ 
is 
$\frac 1 l$-klt since $l(K_{Z_1}+\Delta_1)\sim 0$, 
while 
$(Z_1, S_{Z_1} + G_1)$ 
is dlt.
Hence, 
$(Z_1, \frac 1 2(S_{Z_1} + G_1))$ 
is 
$\frac 1 2$-klt 
and 
$(Z_1, \frac 1 2(\Delta_1+ S_{Z_1} + G_1))$ 
is 
$\frac{1}{2l}$-klt.
\newline
By Step 2.4, we have that
$\vol(K_{Z_1}+ \frac 1 2(S_{Z_1} + G_1)) \leq C'$.
Moreover, the 
$(K_{Z_1}+ \frac 1 2(S_{Z_1} + G_1))$-MMP terminates with a minimal model 
$Z_1 \dashrightarrow Z_2$. 
Since $G_1 \sim \vert f^\ast_{Z_1} (4n+4) A'\vert$ and $A'$ is Cartier and very ample on $X'$, this run of the MMP is over $X'$ and moreover $\vol(K_{Z_1}+ \frac 1 2(S_{Z_1} + G_1)) >0$.
By 
$S_{Z_2}$ 
(resp. 
$G_2$) 
we denote the strict transform of 
$S_{Z_1}$ 
(resp. 
$G_1$) 
on 
$Z_2$.
Hence, Theorem~\ref{mst+fil.thm} implies that the set of all pairs of the form 
$(Z_2, \frac 1 2(S_{Z_2} + G_2))$
that we just constructed 
is strongly log bounded.
\newline
Denoting with  $\Delta_2$ the strict transform of $\Delta_1$ on $Z_2$, it follows that $(Z_2, \Delta_2)$ is a log Calabi--Yau pair such that $l(K_{Z_2}+\Delta_2) \sim 0$.
As $Z_2$ is bounded, there exists a very ample Cartier divisor $H$ and a constant $C_2=C_2(n, d, l)$ such that 
\begin{align*}
|K_{Z_2} \cdot H^{n-1}| \leq C_2.
\end{align*}
As $K_{Z_2}+\Delta_2 \sim_\bQ 0$, it follows that also 
\begin{align*}
\Delta_2 \cdot H^{n-1} \leq C_2, 
\end{align*}
so that $(Z_2, \Delta_2 + G_2)$ is bounded since the coefficients of $\Delta_2$ belong to the set $\{\frac 1 l, \frac 2 l, \dots, \frac{l-1}{l}\}$.
The birational map $Z \dashrightarrow Z_2$ which is the composition of the birational contractions $Z \dashrightarrow Z_1$ (Step 2.2) and $Z_1 \dashrightarrow Z_2$ is in turn a birational contraction.
All these maps are by construction maps over $X'$. 
\newline
As $K_{Z}+\Delta_Z \sim_\mathbb{Q} 0$, it follows from Theorem~\ref{bir.bound.fibr.thm} that there exists a $\mathbb{Q}$-factorial variety $Z'$ with a partial resolution over $X'$ that is isomorphic to $Z$ in codimension one and such that the pair $(Z', \Delta'+G')$ belongs to a bounded family, where we denote by $\Delta'$ the strict transform of $\Delta_Z$ and by $G'$ the pullback on $Z'$ of $G_2$.
As $Z'$ is a partial resolution of $Z_2$ over $X'$, it follows that there exists a morphism $f_{Z'} \colon Z' \to X'$ which is the Iitaka fibration of the semiample divisor $K_{Z'}+\Delta'+ \frac{1}{2}G_1'$. 
Thus Lemma~\ref{bound.morph.lemma} implies that the set $\mathfrak{F}_1$ of 4-tuples $(f_{Z'}, Z', \Delta'+G', X')$ is a log bounded set of morphisms. 

\medskip

{\bf Step 2.6}. {\it In this step, we construct the set $\mathfrak F$ whose existence is claimed in the statement of Step 2}.
\newline
By construction, $(Z', \Delta')$ is isomorphic in codimension one to $(Y^t, \Delta^t)$, by Step 2.1, as it is isomorphic in codimension one to $(Z, \Delta_Z)$. 
As $\Delta^t$ is exceptional over $Y'$, by Remark~\ref{q-fact.degen.rmk} $\Delta'$ is very exceptional for $f_{Z'}$.
Applying Proposition~\ref{MMP.deg.bound.prop}, it follows that the set 
\begin{align*}
    \mathfrak{F}:= &
    \{(f_{Y''}, Y'', 0, X')\;\vert \; \exists (f_{Z'}, Z', \Delta'+G', X') \in \mathfrak{F}' \textrm{ such that } f_{Y''} \colon Y'' \to X' \\
    & \textrm{ is good minimal model for } (Z', (1+\epsilon)\Delta'+G') \textrm{ over } X', \; 0< \epsilon \ll 1\}
\end{align*}
is a log bounded set of morphisms. 

\medskip

{\bf Step 3}. 
{\it In this step we show that $Y'$ is bounded up to flops}.
\newline
We have shown in Step 2 that there exists a log bounded set of morphism $\mathfrak F$ and a 4-tuple $(f_{Y''}, Y'', 0, X') \in \mathfrak F$, together with a birational contraction $Y'' \dashrightarrow Y'$ over $X'$ which is a birational contraction is crepant with respect to $(Y', 0)$.
\newline
Let $D''$ be the exceptional divisor for the map $Y'' \dashrightarrow Y'$.
Remark~\ref{q-fact.degen.rmk} implies that $D''$ is very exceptional for $f_{Y''}$.
By Lemma~\ref{degen.bound.lemma} the set 
\begin{align*}
    \mathfrak{F}_{\rm deg}:=
    & \{(f_{Y''}, Y'', D, X')\;\vert \; (f_{Y''}, Y'', 0, X') \in \mathfrak{F} \textrm{ and }
    \\
    & \; D \textrm{ is very exceptional for } f_{Y''}\}
\end{align*}
is a log bounded set of morphisms, and $(f_{Y''}, Y'', D'', X') \in \mathfrak{F}_{\rm deg}$ by construction. 
Applying Proposition~\ref{MMP.deg.bound.prop} it follows that the set 
\begin{align*}
    \mathfrak{F}_{\rm deg}:= &
    \{(f_{Y'''}, Y''', 0, X')\;\vert \; \exists (f_{Y''}, Y'', D'', X') \in \mathfrak{F}_{\rm deg} \textrm{ such that } \\
    & f_{Y'''} \colon Y''' \to X' \textrm{ is good minimal model for } (Y'', D'') \textrm{ over } X'\}
\end{align*}
is a log bounded set of morphisms.
As $f_{Y'''} \colon Y''' \to X'$ is a relatively good minimal model for $(Y'', D'')$ over $X'$ and $D''$ is the exceptional divisor of the birational contraction $Y'' \dashrightarrow Y'$, it follows that $Y'$ is isomorphic in codimension one to $Y'''$. 
As $Y'''$ belongs to a bounded family, see Remark~\ref{bound.morph.imply.bound.rmk}, the conclusion follows.
\end{proof}

\subsection{Finiteness of index for bases of elliptic Calabi--Yau varieties}

Given an elliptic fibration $f \colon Y \to X$ where $K_Y \sim_{f, \mathbb{Q}} 0$, using the canonical bundle formula, see \S~\ref{sect.cbf}, we can write
\[
K_Y \sim_{\mathbb{Q}} f^\ast (K_{X}+B_{X}+M_{X}). 
\]
\begin{remark}
\label{rmk.coeff.ell.fibr}
Since 
$f$ 
is an elliptic fibration we know that the coefficients of 
$B_{X}$ 
vary in the DCC subset 
$\mathcal R = \{1-\frac1m \ \vert \ m \in \mathbb N_{>0}\} \cup \{ \frac 16, \frac 14, \frac 13, \frac 12, \frac 56, \frac 34, \frac 23\} \subset \mathbb Q \cap (0,1)$, 
cf.~\cite{MR2448282}*{Example 3.1}. 
Moreover, on a sufficiently high birational model of $X$, say, $r \colon X' \to X$, the moduli part $M_{X'}$ becomes semi-ample: more precisely, $|12M_{X'}|$ is base point free, cf.~\cite{MR2448282}*{Example 7.16}.
Hence, we can choose $\overline{M}' \sim 12 M_{X'}$ and define $\overline{M} := r_\ast \overline{M}'$ such that $(X, B_{X} + \frac 1 {12} \overline{M})$ is klt and $12(K_{X}+B_{X}+M_{X}) \sim 12(K_{X} + B_{X} + \frac{1}{12} \overline{M})$.
\end{remark}

\begin{theorem}\label{index.fibr.thm}
Fix positive integers $d, l$.
Let $f \colon Y \to X$ be a contraction of normal projective varieties  such that
\begin{itemize}
    \item $Y$ is of dimension $d$,
    \item $(Y, 0)$ is klt,
    \item $lK_Y \sim 0$, and
    \item $f \colon Y \to X$ is an elliptic fibration.
\end{itemize}
Then there exists $m=m(d, l)$ such that exactly one of the following cases is realized:
\begin{enumerate}
\item[(a)] $mK_X \sim 0$; or
\item[(b)] there exists an effective divisor $B>0$ on $X$, $(X, B)$ is a klt log Calabi--Yau pair and $m(K_X+B) \sim 0$.
\end{enumerate}

The same conclusion holds also if 
$Y$
is only assumed to be proper.

\end{theorem}

\begin{proof}
We use the same notation as in the paragraph before the theorem.
\newline
By Remark~\ref{cbf.lin.eq.rmk} and the hypotheses, we can write the canonical bundle formula in the following form
$
0 \sim lK_Y \sim f^\ast l(K_X+B_X+M_X),
$
and thus
\begin{align*}
0 \sim 12lK_Y \sim f^\ast 12l(K_X+B_X+M_X) \sim f^\ast 12l(K_X+B_X+\frac{1}{12}\overline M).
\end{align*}
As the coefficients of 
$B_{X} + \frac{1}{12} \overline{M}$
vary in the DCC set 
$\mathcal R \cup \{\frac{1}{12}\}$, 
$(X, B_X+\frac{1}{12}\overline M)$ 
is klt, and
$K_{X} + B_{X} + \frac{1}{12} \overline{M} \sim_\mathbb{Q} 0$,
\cite{MR3224718}*{Theorem 1.5} implies that they actually belong to a finite set 
$\mathcal R_0 \subset \mathcal R \cup \{\frac{1}{12}\}$
which only depends on 
$d-1$ 
and 
$\mathcal R$.
Hence there exists a positive integer 
$k=k(\mathcal R_0)$ 
such that 
$kB_{X}$ 
is a Weil divisor.
Hence it suffices to take $m=12kl$ to conclude.

To verify the final part of the statement, it suffices to apply the same argument as the one at the start of the proof of Theorem~\ref{ell.cy.bound.base.thm} to reduce to the projective case.
\end{proof}

Much in the same vein one can prove the following generalization of the above result.

\begin{theorem}
\label{index.fibr.gen.thm}
Fix positive integers 
$d, l$ 
and a bounded set 
$\mathfrak{D}$ 
of K-trivial varieties.
Consider projective varieties 
$X$ 
and contractions 
$f \colon Y \to X$
of normal projective varieties such that
\begin{itemize}
    \item 
$Y$ 
is a klt of dimension 
$d$,

    \item 
$lK_Y \sim 0$, 
and

    \item 
$f \colon Y \to X$ 
is a fibration and the general fibre 
$F$ 
belongs to 
$\mathfrak{D}$.
\end{itemize}
Then there exists a positive integer
$m=m(d, l, \mathfrak{D})$ 
such that exactly one of the following cases is realized:
\begin{enumerate}
	\item[(a)] 
$X$ 
is projective klt and 
$mK_X \sim 0$; 
or

	\item[(b)] 
there exists a generalised klt log Calabi--Yau pair 
$(X, B+M)$ 
and 
$m(K_X+B+M) \sim 0$.
\end{enumerate}

The same conclusion holds also if 
$Y$
is only assumed to be proper.

\end{theorem}

\begin{proof}
By Remark~\ref{cbf.lin.eq.rmk} and the hypotheses, we can write the canonical bundle formula in the following form
\begin{align}
\label{cbf.pf.eqn}
0 \sim lK_Y \sim f^\ast l(K_X+B_X+M_X).
\end{align}
In view of the above linear relation, we can distinguish two cases.

\medskip

{\bf Case 1}: 
$K_X \not \equiv 0$.
\newline
The existence of the klt log Calabi--Yau generalised pair
$(X, B+M)$
is simply a consequence of the canonical bundle formula above~\eqref{cbf.pf.eqn} and Remark~\ref{cbf.rmk}.
By the construction of the boundary divisor 
$B$ 
and ACC of the log canonical threshold, see~\cite{MR3224718}*{Theorem 1.1}, the coefficients of 
$B$ 
belong to a DCC set 
$\mathcal R'$, 
which in turns only depends from $d$.
By~\cite{MR1863025}*{Theorem 3.1}, the Cartier index of the moduli part 
$M'$
on a sufficiently high model 
$X' \to X$
is bounded from above by a positive integer 
$C=C(d, l, \mathfrak{D})$.
Then \cite{bir-zh}*{Theorem 1.6} implies that there exists a positive integer 
$k=k(d, \mathcal R', C)$ 
such that 
$kB_{X}$ 
is a Weil divisor.
Hence,
\begin{align*}
0 \sim 
C!klK_Y \sim 
f^\ast (C!kl(K_{X} + B_{X} + M))
\end{align*}
which concludes the proof of this case.

\medskip

{\bf Case 2}: $K_X \sim_\mathbb{Q} 0$.
\newline
In this case, 
$0 \sim K_Y \sim f^\ast(l(K_X+M_X))$ 
and 
$M_X$ 
is a torsion Weil divisor. 
Moreover, by Corollary~\ref{cbf.triv.cor}, there exists a constant 
$t=t(\mathfrak D)$ 
such that 
$tM_X \sim 0$.
Hence,
\begin{align*}
0 \sim tlK_Y \sim f^\ast tlK_X.
\end{align*}
To verify the final part of the statement, it suffices to apply the same argument as the one at the start of the proof of Theorem~\ref{ell.cy.bound.base.thm} to reduce to the projective case.
\end{proof}

\section{Rationally connected generalised log Calabi--Yau pairs}
\label{rc.gen.cy.sect}
\subsection{Singularities in Mori fibre spaces}
When studying the structure of a Mori fibre space $f \colon X \to Y$ it is natural to ask whether there is any way to control the singularities of the base $Y$ in terms of the singularities of the total space $X$.

A conjecture of Shokurov, later refined by Birkar,~\cite{MR3507921}*{Conjecture 1.2} predicts that one should expect an affirmative answer to the question above.
We have seen in \S~\ref{sect.cbf} that when we have a generalised pair $(X, B+M)$ and a morphism $f \colon X \to Z$ such that $K_X+B+M \sim_{f, \mathbb{Q}} 0$, then we can define a generalised pair $(Z, G+N)$ providing a generalization of the standard form of the canonical bundle formula.
Using the canonical bundle formula for generalised pairs, Birkar's conjecture can be expressed in a more general context.

\begin{conjecture} \label{rel.bab.gen.conj}~\cite{1811.10709}*{Conjecture 2.4}
Let $d$ be a positive integer and $\epsilon$ be a positive real number.
Then there exists a positive real number $\delta = \delta(d, \epsilon)$ such that if $(X, B+M)$ is a generalised pair with a contraction $f \colon X \to Z$ such that 
\begin{itemize}
\item $(X, B+M)$ is $\epsilon$-lc of dimension $d$,
\item $K_X+B+M\sim_{f, \mathbb{Q}} 0$, and
\item $-K_X$ is big over $Z$,
\end{itemize}
then the generalised pair structure $(Z, G+N)$ induced by the canonical bundle formula is $\delta$-lc.
\end{conjecture}

A slightly different version of this conjecture appears in~\cite{MR4436596}*{Conjecture~1.5}.

Birkar proved in~\cite{MR3507921}*{Corollary 1.7} that Conjecture~\ref{rel.bab.gen.conj} holds when 
$\dim X=\dim Z +1$
and 
$M \equiv 0$
-- the latter asusmption implies that the moduli part is numerically equivalent to 
$0$ 
as a b-divisor on 
$Z$, 
which we will use in the proof of 
Theorem~\ref{eps.lcy.3.thm}.
In~\cite{birkar.last}, Birkar has provided a proof of a slightly different version of the above conjecture.

In the Introduction we have already illustrated what the reasons that make Conjecture~\ref{eps.lcy.conj} a central problem in the study of boundedness are. 
Also this conjecture can be extended to the more general setting of generalised pairs.

\begin{conjecture}\label{eps.glcy.conj}
Fix a positive integer $d$ and positive real number $\epsilon$.
Then the set of varieties $X$ such that
\begin{enumerate}
    \item $(X, B+M)$ is an $\epsilon$-lc generalised pair of dimension $d$,
    \item $K_X+B+M \sim_\mathbb{Q} 0$, and 
    \item $X$ is rationally connected
\end{enumerate}
is bounded.
\end{conjecture}

We will denote by Conjecture~\ref{eps.glcy.conj}$_{\leq d}$ the statement of Conjecture~\ref{eps.glcy.conj} for all generalised pairs of dimension at most $d$.
We can prove the following conditional step towards the solution of the conjecture above.

\begin{theorem}
\label{thm.conditional.bound}
Assume that 
Conjecture~\ref{eps.glcy.conj}$_{\leq d}$ 
and that Conjecture~\ref{rel.bab.gen.conj} hold.
Then the set of varieties $X$ such that
\begin{enumerate}
    \item $(X, B+M)$ is an $\epsilon$-lc generalised pair of dimension $d+1$,
    \item $K_X+B+M \sim_\mathbb{Q} 0$,
    \item $K_X\not \equiv 0$, and 
    \item $X$ is rationally connected
\end{enumerate}
is bounded up in codimension one.
\end{theorem}

\begin{proof}
We start by proving the following reduction.

\medskip

{\bf Claim}.
{\it 
To prove Theorem~\ref{thm.conditional.bound}, it suffices to show boundedness in codimension one for the set of varieties
$X$ 
that satisfy conditions (1)-(4) in the statement of the theorem and that are moreover endowed with a 
$K_X$-negative 
Mori fibre spaces structure 
$f \colon X \to Z$.
}

\begin{proof}[Proof of the Claim]
Since 
$K_X \not \equiv 0$, 
then we can run the 
$K_X$-MMP
$\pi \colon X \dashrightarrow X'$ 
which terminates with a
$K_{X'}$-negative
Mori fibre space structure 
$f' \colon X' \to Z$.
Let us observe that, by definition of a run of the MMP, $\pi$ is a birational contraction and since it is
$K_{X}$-negative, 
then 
$K_{X'} \not \equiv 0$.
Defining 
$B':= \pi_\ast B$,
$M':= \pi_\ast M$, 
then 
$(X',B'+M')$
is a generalized log Calabi--Yau pair,
$\pi$ 
is crepant for 
$(X,B+M)$ 
and
$(X',B'+M')$, 
and 
$(X', B'+M')$
also satisfies condition (1)-(4) in the statement of the theorem.
Hence, if the set of $X'$ just constructed is bounded in codimension one, then by Corollary~\ref{bir.bound.fibr.cor} the same must hold for the original set of varieties $X$ satisfying assumptions (1)-(4) of the statement the theorem.
\end{proof}

Hence, we will assume that $X$ is endowed with a Mori fibre space structure $f \colon X \to Z$. 
The target variety $Z$ will also be rationally connected, being the image of a rationally connected variety.
The canonical bundle formula, see \S~\ref{sect.cbf}, implies that there is a generalized log Calabi--Yau structure $(Z, G+N)$ on $Z$.
Moreover, by Conjecture~\ref{rel.bab.gen.conj} $(Z, G+N)$ is $\delta$-lc for some $\delta=\delta(d+1, \epsilon)$.
Hence, by applying Conjecture~\ref{eps.glcy.conj}, $Z$ belongs to a bounded family.
Thus, the conclusion follows by applying Theorem~\ref{bir.bound.fibr.thm}, since
$f \colon (X, B+M) \to Z$ 
is a 
$(d+1, r, \epsilon)$-Fano 
type fibration, for some 
$r \in \mathbb N_{>0}$ 
that is given by the boundedness of 
$Z$ that we just proved.
\end{proof}

\section{Proofs of the theorems and corollaries}
\label{pf.sect}

\begin{proof}[Proof of~\ref{delta.lcy.thm}]
We start by proving the following reduction.

\medskip

{\bf Claim}.
{\it 
To prove Theorem~\ref{delta.lcy.thm}, it suffices to show that the set of pairs
$(X, B)$ 
that satisfy conditions (1)-(3) in the statement of the theorem and that are moreover endowed with a tower of Mori fibre spaces
\begin{align}
\label{tower.eq.pf}
\xymatrix{
 (X, B)=(X_0, B_0) \ar[r] & X_1 \ar[r] & \dots \ar[r] & X_{s-1} \ar[r] & X_s=Z
 }
\end{align}
satisfying conditions (i)-(iii) of Theorem ~\ref{structure},
is log bounded up to flops.
}
\begin{proof}[Proof of the Claim]
Theorem~\ref{structure} implies that, given a log pair 
$(X, B)$ 
satisfying conditions (1)-(3) in the statement of Theorem~\ref{delta.lcy.thm}, there exists a log pair 
$(X', B')$ 
such that 
\begin{itemize}
    \item 
$(X', B')$ 
also satisfies conditions (1)-(3) in the statement of Theorem~\ref{delta.lcy.thm},
    \item 
there exists a birational contraction 
$\pi \colon X \dashrightarrow X'$, 
and,
    \item 
$(X', B')$ 
is moreover endowed with a tower of Mori fibre spaces as in~\eqref{tower.eq.pf}.
\end{itemize}
As 
$(X, B)$ 
and 
$(X', B')$ 
are log Calabi--Yau pairs, then 
$K_X+B=\pi^\ast (K_{X'}+B')$ 
and 
$\pi$ 
is crepant -- here, 
$\pi^\ast(K_{X'}+B')$
is well-defined even though 
$\pi$ 
is just a birational map since Theorem~\ref{structure} guarantees that 
$\pi$ 
is a birational contraction. 
Hence, 
if we assume that the set of pairs 
$(X', B')$
just constructed is log bounded up to flops, then Corollary~\ref{bir.bound.fibr.cor} implies that also the set of pairs 
$(X, B)$ 
satisfying conditions (1)-(3) in the statement of the theorem is log bounded up to flops.
\end{proof}
Hence, we will assume that the pairs $(X, B)$ are endowed with a tower of Mori fibre spaces as in~\eqref{tower.eq.pf}. 
If $(X, B)$ is not of product type, then $\dim Z=0$ and the conclusion follows from Corollary~\ref{logCY.fibred.cor}.
If $\dim Z >0$ and $K_Z \sim_\mathbb{Q} 0$, then Theorem~\ref{mfs.tow.CYend.thm} implies that there exists $m=m(d, l)$ such that $mK_Z \sim 0$.
Then, by Theorem~\ref{rc.cy.log.bound.thm}, $Z$ is bounded up to flops.
That is, there exists a klt variety $Z'$ isomorphic to $Z$ in codimension one which belongs to a bounded family.
By Corollary~\ref{bir.bound.fibr.cor}, we can assume that $Z'$ is $\mathbb{Q}$-factorial.
As, $Z' \dashrightarrow Z$ is an isomorphism in codimension one of projective varieties, it is also a birational contraction.
Thus, we can apply Proposition~\ref{bir.contr.lcy.prop} to the Mori fibre space $X_{s-1} \to Z$ in~\eqref{tower.eq.pf} and obtain a commutative diagram 
\begin{align*}
\xymatrix{
X'_{s-1}  \ar@{-->}[r] \ar[d]&
X_{s-1} \ar[d]\\
Z' \ar@{-->}[r] &
Z
}
\end{align*}
where the horizontal arrow $X'_{s-1} \dashrightarrow X_{s-1}$ is an isomorphism in codimension one of $\mathbb{Q}$-factorial projective varieties.
As all horizontal arrows in the diagram are isomorphisms in codimension one of $\mathbb{Q}$-factorial projective varieties, it follows that $\rho(X_{s-1} /Z)=\rho(X_{s-1}' /Z')=1$; 
hence, $X_{s-1}' \to Z'$ is a Mori fibre space.
As the map $X'_{s-1} \dashrightarrow X_{s-1}$ is in turn a birational contraction, we can apply Proposition~\ref{bir.contr.lcy.prop} to the MFS $X_{s-2} \to X_{s-1}$ in~\eqref{tower.eq.pf} and obtain a commutative diagram 
\begin{align*}
\xymatrix{
X'_{s-2}  \ar@{-->}[r] \ar[d]&
X_{s-2} \ar[d]\\
X'_{s-1} \ar@{-->}[r] &
X_{s-1}
}
\end{align*}
where the horizontal arrow $X'_{s-2} \dashrightarrow X_{s-2}$ is an isomorphism in codimension one of $\mathbb{Q}$-factorial projective varieties.
\newline
By inductively applying Propositions~\ref{bir.contr.lcy.prop}, as we just did, we construct a pair $(X', B')$ isomorphic in codimension one to $(X, B)$ and endowed with a tower of Mori fibre spaces
\begin{align}
\label{tower.eq.pf.1}
\xymatrix{
 (X', B') \ar[r] \ar@{-->}[d]& 
 X'_1 \ar[r] \ar@{-->}[d]& 
 \dots \ar[r] & 
 X'_{s-1} \ar[r] \ar@{-->}[d]& 
Z'\ar@{-->}[d] \\
 (X, B) \ar[r] & 
 X_1 \ar[r] & 
 \dots \ar[r] & 
 X_{s-1} \ar[r] & 
 Z,
 }
\end{align}
where all the vertical arrows are isomorphisms in codimension one and $\rho( X'_{i}/ X'_{i+1})=1$, for $i=0, \dots, s-1$.
\newline
The conclusion then follows by applying Theorem~\ref{birk.tower.thm} to the pair $(X', B')$ and the tower of morphisms in~\eqref{tower.eq.pf.1}.
\end{proof}

\begin{proof}[Proof of~\ref{ell.cy.thm}]
Since $Y$ is a simply connected Calabi--Yau, by Corollary~\ref{rc.base.cor} $X$ is rationally connected. 
Thus, Theorem~\ref{index.fibr.thm} implies that there is a choice of an effective divisor $B \geq 0$ (where we allow the possibility of $B=0$), such that the pair $(X, B)$ is klt, log Calabi--Yau, and the torsion index of $K_X+B$ is bounded.  
By Theorem~\ref{delta.lcy.thm}, the pair $(X, B)$ is log bounded up to flops.
By Propositions~\ref{small.qfact.lcy.prop} and~\ref{bir.contr.lcy.prop}, up to substituting $Y$ with an elliptic terminal Calabi--Yau $Y' \to X'$ with $Y'$ isomorphic to $Y$ in codimension one, we can assume that the base of the elliptic fibration is actually bounded.
The proof then follows immediately from Theorem~\ref{ell.cy.bound.base.thm}. 
\end{proof}

\begin{proof}[Proof of~\ref{ell.cy.thm.gen}]
As $Y \to X$ is an elliptic fibration, Theorems~\ref{index.fibr.thm}  implies that there is a choice of an effective divisor $B \geq 0$ (where we allow the possibility of $B=0$), such that the pair $(X, B)$ is klt, log Calabi--Yau, and the torsion index of $K_X+B$ is bounded
by a constant $m=m(l,d)$.
Hence, as $X$ is rationally connected, the pair $(X, B)$ is log bounded up to flops, by Theorem~\ref{delta.lcy.thm}.
By Propositions~\ref{small.qfact.lcy.prop} and~\ref{bir.contr.lcy.prop}, up to substituting $Y$ with an elliptic K-trivial variety $Y' \to X'$ with $Y'$ isomorphic to $Y$ in codimension one, we can assume that the base of the elliptic fibration is actually bounded.
The proof then follows immediately from Theorem~\ref{ell.cy.bound.base.thm}, since $Y$ is $\frac 1 l$-lc.
\end{proof}

\begin{proof}[Proof of~\ref{eps.lcy.3.thm}]
The case where $K_X \equiv 0$ and $X$ is a variety with strictly klt singularities follows from~\cite{1904.09642}*{Theorem 1.6}.
Hence, we can assume that $K_X \not \equiv 0$.
We can run a $K_X$-MMP
\begin{align}
\label{eqn:diag.mmp}
\xymatrix{
X=:X_0 \ar@{-->}[r]^{\pi_0} &
X_1 \ar@{-->}[r]^{\pi_1} &
X_2 \ar@{-->}[r]^{\pi_2} &
\dots \ar@{-->}[r]^{\pi_{n-2}} &
X_{n-1} \ar@{-->}[r]^{\pi_{n-1}} &
X_n
}
\end{align}
which terminates with a Mori fibre space
$f \colon X_n \to Z$.
We denote by 
$\pi \:= \pi_{n-1} \circ \pi_{n-2} \circ \dots \circ \pi_1 \circ \pi_0$.
The dimension of $Z$ is either $0, 1$, or $2$; 
moreover, as 
$X$ 
is RC, so is 
$Z$.
We denote by 
$B_n +M_n$
the strict transform of
$B +M$ 
on 
$X_n$. 
The pair
$(X_n, B_n +M_n)$ 
is 
$\epsilon$-lc 
generalised log Calabi--Yau,
and 
$B_n +M_n \not \equiv 0$,
since the 
$K_X$-MMP 
we ran was 
$B+M$-positive
as 
$K_X+B+M\sim_{\mathbb{R}}0$.
We first show that the set of varieties
$X_n$
obtained by running the 
$K_X$-MMP 
as in~\eqref{eqn:diag.mmp} 
is bounded.
\newline
If 
$\dim Z=0$, 
then 
$\rho(X_n)=1$
and, by the above observations, 
$X_n$ 
is a 
$\epsilon$-lc
Fano threefold.
Thus, the set of all varieties
$X_n$ 
obtained with the above construction and $\dim Z=0$ is bounded by
\cite{Bir16b}.
\newline
If $\dim Z=1$, 
then 
$Z = \mathbb P^1$, 
which is bounded.
Thus, 
$f \colon X_n \to Z$
is a
$(\epsilon, 3, 1)$-Fano 
fibration and Theorem~\ref{bir.bound.fibr.thm} implies the boundedness of the set of all varieties
$X_n$ 
obtained with the above construction when 
$\dim Z=1$.
\newline 
If 
$\dim Z=2$, 
then we prove the following claim.

\medskip

{\bf Claim}. 
{\it $Z$ is a $\delta$-lc Fano surface, for some positive $\delta=\delta(\epsilon)$}. 
\begin{proof}[Proof of the Claim]
If $M_n \sim_{Z, \mathbb{R}} 0$ then $K_{X_n}+B_n \sim_{Z, \mathbb{R}} 0$, hence~\cite{MR3507921}*{Corollary 1.7} implies that $Z$ is $\delta$-lc.
If $M$ is relatively ample over $X_n$ then letting $H_Z$ be an ample Cartier divisor on $Z$,
the results follows again from 
\cite{MR3507921}*{Corollary~1.7}
by adding considering the log divisor 
$K_{X_n}+B_n+M_n+tf^\ast H_Z$
for 
$t \gg 0$.
\end{proof}
\noindent 
The claim and the solution to the BAB conjecture for surfaces~\cite{MR1298994}*{Theorem 6.8} together imply that
the set of all surfaces
$Z$ 
just constructed is bounded. 
In turn, Theorem~\ref{bir.bound.fibr.thm} implies that the set of all varieties
$X_n$ 
obtained with the above construction and 
$\dim Z=2$
is bounded.
The proof of this last case then concludes by applying Corollary~\ref{bir.bound.fibr.cor}, since 
$\pi$
is a birational contraction which is crepant for 
$(X, B+M)$ and $(X_n, B_n+M_n)$.
\end{proof}

\begin{proof}[Proof of~\ref{hodge}] 
The result follows from the same argument as in the proof of \cite{DCS}*{Corollary~1.2} combined with Theorem~\ref{ell.cy.thm}. 
\end{proof}

\bibliographystyle{plain}

\end{document}